\documentclass[12pt]{article}

\usepackage{latexsym,amssymb,amsmath}
\usepackage[dvips]{graphicx}
\newcommand{\C}{\mathbb C}
\newcommand{\Z}{\mathbb Z}
\newcommand{\N}{\mathbb N}
\newcommand{\R}{\mathbb R}
\newcommand{\Q}{\mathbb Q}
\newcommand{\F}{\mathbb F}

\newcommand{\sma}{\left(\begin{array}}
\newcommand{\fma}{\end{array}\right)}
\newcommand{\injects}{\hookrightarrow}

\newtheorem{lem}{Lemma}[section]
\newtheorem{defn}[lem]{Definition}
\newtheorem{ex}[lem]{Example}
\newtheorem{co}[lem]{Corollary}
\newtheorem{thm}[lem]{Theorem}
\newtheorem{prop}[lem]{Proposition}

\newenvironment{proof}{\textbf{Proof.}}{\newline\hspace*{\fill}{$\Box$}\\}

\begin{document}
\title{Groups acting faithfully on trees and properly on products of trees}
\author{J.\,O.\,Button\\
Selwyn College\\
University of Cambridge\\
Cambridge CB3 9DQ\\
U.K.\\
\texttt{j.o.button@dpmms.cam.ac.uk}}
\date{}
\maketitle
\begin{abstract}
We examine the question of which finitely generated groups act properly
on a finite product of simplicial trees, considering both arbitrary trees 
and where all trees are locally finite. In the second case we present
evidence in favour of hyperbolic surface groups having such an action.
However we also present evidence that many RAAGs do not admit such an
action and we give an example of a virtually special group which does
not act properly preserving factors on any finite product of locally
finite trees, even though it does so on a product of three trees without
the local finiteness condition.   
\end{abstract}

\section{Introduction}
Given an abstract group $G$ and a geodesic metric space $X$,
it is of interest to know whether $G$ acts properly and 
cocompactly by isometries on $X$. 
For instance (\cite{delh}, \cite{bh})
$G$ is finitely generated if and only it acts as such on some
space $X$ and finitely presented if and only if there is some
simply connected space $X$ with such an action of $G$.
However for a given $X$, or indeed if $X$
comes from a particular class of geodesic metric spaces,
this can impose strong restrictions on the group; in particular $G$ 
must be quasi-isometric to $X$. 

One feature of this restriction is that if we want to understand
the subgroup structure of $G$ then any finite index subgroup $H$
of $G$ will also have such an action on $X$ but an arbitrary
subgroup might not. Also of concern is that if $G$ is contained
in an overgroup $L$ with finite index then a suitable action of
$G$ on $X$ might not extend to one of $L$. (For instance,
consider the free group $F_2$ acting on its Cayley graph, the
4-regular tree $T_4$, and a subgroup $F_4$ of index 3 also
acting nicely on $T_4$. If we now go
up to the free product $C_5*C_5$ which contains $F_4$ with
index 5, any action of this group on $T_4$ by isometries
fixes every point.) 
This means the property of ``acting properly and cocompactly on $X$
for $X$ a geodesic metric space in some well behaved class'' might not 
give us a commensurability invariant, even though of course
commensurable finitely generated groups are quasi-isometric.
 
Our approach in this paper is to keep the properness of the action
but to drop the cocompactness. However in the absence of the cocompactness
condition, a proper action can have different definitions in the literature
(and is sometimes used interchangeably with a discrete action), so
we will need to consider which of these definitions we will use,
particularly when $X$ is not a proper metric space. Having said that,
any reasonable definition of what it means for $G$ to act properly
on $X$ immediately holds for an arbitrary subgroup of $G$ by
restriction. As for extending to supergroups $L$ where $G$ has finite 
index $i$ say in $L$, the example above shows we cannot expect to do this for
the same space $X$. However, suppose we have a class $\mathcal C$
of geodesic metric
spaces which is closed under taking finite direct products (say with
the $\ell_1$ or $\ell_2$ product metric). On forming the direct product $P$
of $i$ copies of $X$, we can induce the action of $G$ up to that of $L$
where $L$ will permute the copies of $X$ in the product. This action
will still be by isometries and will be proper if the action of $G$
on $X$ is (though cocompactness is not preserved in general). Thus
if we have this closure property for $\mathcal C$ we find that 
``acting properly on a geodesic metric space $X$ in the class $\mathcal C$''
becomes a commensurability invariant, as well as being preserved under
taking arbitrary subgroups.  

What then might we take for our class $\mathcal C$ when studying the 
question of which groups $G$ act properly on spaces in $\mathcal C$?
If $\mathcal C$ is too general then we might not obtain any
information about our group $G$. For instance the Groves - Manning
combinatorial 
horoball construction (see \cite{gr}) shows that every countable group
$G$ acts properly by isometries on a locally finite hyperbolic graph
and it seems to be open whether every countable group acts properly
on some CAT(0) space.

On the other hand, if our class $\mathcal C$ of spaces is too restricted
then we might find that our class of groups acting properly on $\mathcal C$
is restricted as well. For instance if $\mathcal C$ consists of all
simplicial trees then by work of Serre and others, the finitely
generated groups acting properly on some space in $\mathcal C$ are
exactly the finitely generated virtually free groups. Moreover it
does not matter in this result whether our class of trees are all
locally finite or not.

Examples of classes $\mathcal C$ where the question of which groups
act properly on spaces in $\mathcal C$ has been investigated recently
include the class of CAT(0) cube complexes, where one might or might
not restrict to finite dimensional and/or locally finite spaces, and
(in \cite{bbk}, \cite{bbk2}) finite products of quasi-trees.  

In this paper our class $\mathcal C$ of spaces are finite products
of simplicial trees (thus obtaining a commensurability invariant
as mentioned above), where we consider the case of general trees (in
Section 5) separately from the case of locally finite trees (in Section 6).
Lemma 6.1 shows that for finitely generated groups, the question of
having a proper action where the factor trees are locally finite is
equivalent to the trees having uniformly bounded degree or
being regular trees of some finite degree. Before this in Section 4,
we look at the various definitions in use of a proper/discrete action
and show for clarity in Corollary 4.2 that all of these are equivalent
when acting on a finite product of locally finite trees. However
for trees in general, Theorem 4.1 shows that most, but not all definitions
are equivalent (acting metrically properly is too strong and acting
discretely is too weak).

However before we look at proper actions on products of trees, Sections 
2 and 3 consider how a group acts on a single tree. 
In the finitely generated case, we know that Bass - Serre
theory tells us which groups act properly on trees and which groups
act on some tree without a global fixed point. Here our emphasis is on
which groups have a faithful action on some tree, where we consider
actions both with and without a global fixed point. In fact these
questions are not interesting for general trees, but in the locally
finite case we give in Theorem 2.2 the characterisation of which
finitely generated groups have a faithful action on some locally finite
tree. In Section 3 we look at the case where the tree is of uniformly
bounded degree, or equivalently (in terms of groups possessing a
faithful action) a regular tree of finite degree. In
Theorem 3.2 (which treats the case of a global fixed point but deals
with arbitrary groups) and
Corollary 3.3 (for finitely generated groups but general actions) we
state which groups have a faithful action on some regular tree. In
particular (Example 3.5)
there are finitely generated groups with a faithful action on some
locally finite tree but no faithful action on any uniformly bounded
tree, in contrast to the case of proper actions.

Defining $T_d$ to be the regular tree where every vertex has degree $d$,
saying $G$ acts faithfully on $T_d$ is the same as saying $G$ is a
subgroup of the group $Aut(T_d)$ of simplicial automorphisms. For degree 
$d'>d\geq 3$, we would expect that $Aut(T_d)$ and
$Aut(T_{d'})$ are never isomorphic as abstract groups and indeed this
was shown in \cite{z}, with other proofs in \cite{mol} and \cite{bslb}.
However we can give in Theorem 3.6 an extremely quick and easy proof
of this result, simply by finding for each $d'$ a well known finite group 
which embeds in $Aut(T_{d'})$ but not $Aut(T_d)$. As a variation, we can
consider the regular rooted trees $R_d, R_{d'}$ where the subscript
is the degree of the root vertex, with every other vertex having degree
one higher, and the abstract groups $Aut(R_d)$, $Aut(R_{d'})$ (where
every automorphism will fix the root vertex). Exactly the same argument
of finding a small finite subgroup of $Aut(R_{d'})$ which is not
contained in $Aut(R_d)$ works to show that these groups are not
isomorphic, except when $d'=4$ and $d=3$. This gives rise to a strange
phenomenon proved in Theorem 3.8, which is that $Aut(R_3)$ is a subgroup
of $Aut(R_4)$ and (rather less obviously) $Aut(R_4)$ is a subgroup
of $Aut(R_3)$. The idea behind this is that any group with a subgroup of index
4 also has one of index 2 or 3.

Moving on from actions on trees to actions on products of trees, we discuss 
definitions and generalities of such actions in Section 4. Our aim in Sections
5 and 6 is to demonstrate that acting properly on a finite product of trees
is not so unusual a property for a group to have, whereas requiring a proper
action where all the trees are locally finite is much more restrictive. It
is known that any virtually special group (a group with a finite index
subgroup that embeds in a RAAG, where we take RAAGs to have defining
graphs which are finite) has the former property but we give a quick
and basic proof in Theorem \ref{rg}. However a surprise is in store when
we ask what is the minimum number of trees we can have in a proper action.
Although (excluding free groups) the 2 dimensional RAAGs are the ones with
defining graphs which are triangle free, under various definitions such
as geometric or cohomological dimension,
we show by combining Corollary \ref{chr} and
Proposition \ref{odd} that the RAAGs acting properly on a product of two
trees are precisely those whose defining graph has chromatic number (1 or)
2, or equivalently those graphs which do not contain an odd length closed
path. This immediately implies Corollary \ref{sbgr} which states that a
RAAG whose defining graph contains an odd length closed path cannot embed
in a RAAG whose graph does not.

The question of which groups possess
proper actions on finite products of locally finite trees seems much
more mysterious. For instance we have direct products of free groups,
Burger - Mozes groups (these two examples also act cocompactly), certain
wreath products and subgroups of these examples. But what about an
example of a word hyperbolic group with such an action but which is not
virtually free? The obvious example to try would be the fundamental group
$S_g$ of a closed orientable hyperbolic surface $\Sigma_g$ with genus 
$\geq 2$. In fact exactly this question was raised in \cite{flss} which
gives partial results and which was the motivation for much of this work.
Before examining this case though, we look in Section 6 at some groups
containing $\Z\times\Z$. For instance, we can ask: does every RAAG act
properly on a finite product of locally finite trees? (Many RAAGs are
known to contain surface subgroups so a yes answer for this question
would certainly answer the previous question as well.) However we
present evidence for a negative answer. In Proposition \ref{thrtr} and
Theorem \ref{noac} we find a related group, a CAT(0) and virtually
special group $G$ of the form $F_2\rtimes\Z$, which does not act
properly preserving factors on any finite product of locally finite
trees, even though it acts properly preserving factors on a product
of 3 trees if we remove the local finiteness condition. There is
a particular index 2 subgroup $N$ of $G$ where if it were shown
that $N$ does not act properly preserving factors on any finite product 
of locally finite trees then this would imply (by Corollary \ref{noacc})
that most RAAGs have no proper action on a finite product of locally finite
trees, whether or not the action preserves factors. However we leave this
particular question open.

The final two sections look at the existence of a proper action
for our surface groups $S_g$
and we present some evidence in favour, by
considering how these groups can act on a single locally finite tree.
Corollary \ref{injcloc2} shows that a necessary condition for such an
action is that for any non identity element $\gamma$ in $S_g$ (or
indeed any torsion free group which does not contain $\Z\times\Z$),
there is an action of $S_g$ on some locally finite tree which is
minimal, faithful and such that $\gamma$ acts as a hyperbolic element.
Note that without the faithful condition we could use the residual
freeness of $S_g$ to obtain a suitable action on the Cayley graph
of a finitely generated free group, whereas without the minimal
condition we could take the same action and convert it into a
faithful action on a different locally finite tree using the
techniques of Section 2.

It is pointed out in \cite{flss} that if we could find an embedding
of $S_g$ in $SL(2,\F)$ for $\F$ a global field of positive characteristic
(a finite degree field extension of $\F_p(x)$ for some prime $p$) then
$S_g$ would act properly on a finite product of locally finite trees, by
taking the Bruhat - Tits tree associated to a finite number of
discrete valuations $v$ on $\F$. Given such a global field $\F$ and 
some $v$, we have the local field $k$ obtained by taking
the completion of $\F$ with respect to $v$. Now it is clear
that for any prime $p$, the group $S_g$ embeds in $SL(2,k)$ for some
local field $k$ of characteristic $p$. For instance, using \cite{lubseg}
Window 8 Theorem 1 which was originally due to 
Malceev, the two facts that a finite
rank free group has a 2 dimensional faithful linear representation
in characteristic $p$ and
that $S_g$ is fully residually free implies that $S_g$ has such a
representation too. As $S_g$ is finitely generated,
we can take $K$ to be finitely generated over its prime subfield $\F_p$
which means that $K$ is a finite extension of
$\F_p(t_1,\ldots ,t_d)$ where obviously $d>0$ and $t_1,\ldots ,t_d$
are algebraically independent elements. Thus if $k$ is any
local field containing $\F_p(x)$, say $\F_p((x))$, and we set $t_1=x$
then the uncountability of $k$ means that the field
$\F_p(t_1,\ldots ,t_d)$ will embed in $k$ and moreover any finite
extension of $\F_p(t_1,\ldots ,t_d)$ will embed in some finite extension
of $k$ which will also be a local field.
 
Thus we can view the question of whether $S_g$ embeds in $SL(2,\F)$ for
$\F$ a global field (which implies the existence of  a proper action)
as the top question in a series of questions as to how ``economical''
we can take our field $K$ to be, since this is the $d=1$ case. As this
is unknown, we can instead ask how small we can take $d$ in a faithful
representation of $S_g$ and also how small the degree of our field extension
needs to be. In \cite{flss} it was shown that 
for every prime $p$ at least 5, there is a faithful embedding
of $S_2$ (and hence $S_g$) in $PGL(2,K)$ where $K$ is a finite extension of 
$\F_p(x,y)$, thus we can take $d=2$ if $p\neq 2,3$. Here we remove the
condition on $p$ and the need for a finite extension in that we provide in
Theorem \ref{mat} and Corollary \ref{mat2} a completely explicit
embedding of $S_2$ (and hence $S_g$ for any $g\geq 2$) in
$SL(2,\F_p(x,y))$ for any prime $p$. (In fact, as is presumably usual, the
four matrices provided in Theorem \ref{mat} are independent of $p$ and
work for every odd prime $p$, whereas $p=2$ requires a different
representation which is given in Corollary \ref{mat2}.) As this is the
most ``economical'' field possible for a 2 dimensional representation of
$S_g$ in positive characteristic short of actually establishing
a proper action on a finite product of locally finite trees,
we feel that this provides evidence in favour of
the existence of such an action. Moreover our faithful representation
of $S_g$ in $SL(2,\F_p(x,y))$ provides in Corollary \ref{moreacc}, for 
any non identity element $\gamma$ in $S_g$, a minimal faithful action
of $S_g$ on a locally finite tree in which $\gamma$ acts hyperbolically,
thus providing further evidence in favour.

Our faithful representation of $S_2$ is obtained by using results in
\cite{con} on how to determine whether a pair of elements in the
automorphism group of a locally finite tree generate a discrete and
faithful copy of $F_2$. We then combine this with a result in \cite{sha}
on constructing faithful linear representations of an amalgamated free
product, given faithful linear representations of the factor groups.
In order to apply this result, we require faithful representations
of the free group on $a,b$ say in $SL(2,\F)$ for suitable fields $\F$
where the commutator $aba^{-1}b^{-1}$ is a diagonal matrix. The results
of our calculations
are stated in Theorem \ref{matfrm} and can be taken on trust if so desired,
but the argument explaining how to obtain the forms of these matrices 
is included here, though it was felt best to relegate this part to the
Appendix.

\section{Groups acting faithfully on locally finite 
trees}

For us, all trees in this paper are simplicial trees, defined 
in the standard combinatorial way as in \cite{ser}, but then
regarded as metric spaces by equipping them with the resulting path
metric. They might or might not be locally finite trees; indeed a theme
here will be examining the differences between the two cases. 
(Note that when the tree is not locally finite, the topology induced
by the path metric is not the same as the CW topology.)

Given any tree
$T$, we will use $Aut(T)$ to denote
the group of simplicial automorphisms of $T$ and these will be 
isometries of $T$. Moreover, saying that an abstract group $G$ acts on a tree 
$T$ will also mean that $G$ acts on $T$ by
simplicial automorphisms, but it need not imply that the action is
faithful; in fact one of our aims here is to
see which groups do possess faithful actions on various trees.
We will be vague as to whether our group $G$ acts on $T$ with or without
edge inversions, which can be justified by the following reason: when
we refer to $Aut(T)$ for a particular tree $T$, this will always
include those simplicial automorphisms which invert an edge. However when we
are considering an abstract group $G$, our focus will be on whether
or not there is an action of $G$ on {\it some} tree $T$ that
satisfies suitable conditions. Thus if $G$ does invert an edge when
acting suitably on $T$ then it would act suitably without edge inversions 
and still by isometries on the barycentric subdivision of $T$.
With this in mind, we say that an action of a group $G$ on a tree $T$
is free if no group element
except the identity has a fixed point in $T$ (equivalently no vertex
or midpoint of an edge in $T$ is fixed, or again equivalently no
vertex of the barycentric subdivision is fixed)
in which case the action will of course
be faithful. A theorem of Serre (\cite{ser} Proposition 15 and Theorem 4) 
states that a group $G$ acts freely 
on some tree $T$ if and only if it is
a free group. Note that in this statement there
is no restriction that $G$ is finitely generated, nor that $T$ is a
locally finite tree. 

But what about other actions of a group $G$ on a tree $T$? Of course 
in Bass - Serre
theory we have an important distinction between actions with a global fixed
point (sometimes called trivial actions, but here we reserve that term for when
every element acts as the identity) and those without. Indeed a main
feature of Bass - Serre theory (\cite{ser} Theorem 15) is the result that a
finitely generated group has an action on some tree without a global fixed
point if and only if it splits non trivially as an amalgamated free
product or HNN extension.

Here though we are also interested in actions of a group on a tree which do
have a global fixed point. Of course the trivial action always exists, so
we can start by asking which groups act faithfully on some tree with a global
fixed point. But actually any group $G$ does, as one can take a 
vertex for each element of $G$ along with a root vertex $v_0$ joined to all
other vertices (but with no other edges) and then let $G$ fix $v_0$ but
act on the other vertices by self multiplication (we call this the star
graph construction).

On now restricting to finitely generated groups $G$, whereupon we know the
condition for existence of an action of $G$ on some tree without a global
fixed point, our next question might therefore be: when does $G$ have
an action on a tree without a global fixed point which is faithful? 
However again
this turns out not to be an interesting question because the answer is
the same as without the condition of faithfulness, which can be seen
for instance by using the construction given later in Theorem \ref{fthtre}.

Instead we consider the case where all our trees are locally finite, 
whereupon we are able to characterise those finitely generated groups with 
such actions as above. In the case of a global fixed point, the following 
lemma is folklore: 

\begin{lem} \label{resf}
A countable group $G$ acts faithfully on some locally finite tree $T$ 
with a global fixed point
if and only if $G$ is residually finite.
\end{lem}
\begin{proof}
We consider the global fixed point $v_0$ to be the root of the tree $T$.
Given any vertex $v\in T$, we have that the orbit of $v$ under any group
$G$ acting on $T$ and fixing $v_0$ 
is finite because $T$ is locally finite, so the stabiliser
$S_v$ of any $v\in T$ has finite index in $G$. But if the action of $G$ is
also faithful then for any non identity element $g\in G$ we have
some vertex $v$ with $g(v)\neq v$ and so $g\notin S_v$.

Now suppose that $G$ is countable and residually finite. By enumerating
$G\setminus\{id\}=\{g_1,g_2,\ldots \}$, we can find a chain
$G=G_0> G_1 >G_2>\ldots $ of finite index subgroups $G_i\leq G$ having
trivial intersection. We then inductively create our tree $T$
by starting with a root vertex $v_0$ corresponding to $G_0$. 
The vertices at the $n$th level of $T$ are the cosets of $G_n$
in $G$, with the coset $gG_n$ joined by an edge to the coset $\gamma G_{n-1}$
on the previous level if and only if $gG_n\leq \gamma G_{n-1}$, which happens
if and only if $g\in \gamma G_{n-1}$. Then left multiplication of $G$ on
these cosets of the various subgroups $G_i$ gives us an action of $G$
on the tree $T$ with fixed point $v_0$. Moreover if $g\notin G_n$ then
$g$ moves the vertex $G_n$ to $gG_n\neq G_n$, so this action is
faithful.
\end{proof}

Note: by the same means, we can immediately see that a group acts
non trivially on some locally finite tree with a global fixed point
if and only if it has a proper finite index subgroup.

In line with the list above of questions about group actions
on arbitrary trees, we next ask: which finitely generated groups act 
faithfully on some locally finite tree without a global fixed point? We 
can now answer this in full, thus in combination with the above
we have a complete characterisation of when a finitely generated group $G$
embeds in $Aut(T)$ for some locally finite tree $T$.

\begin{thm} \label{fthtre}
A finitely generated group $G$ has a faithful action on some locally finite
tree without a global fixed point
if and only if $G$ can be expressed as the fundamental group of a 
non trivial finite 
graph of groups with all vertex groups residually finite (but not
necessarily finitely generated) and all edge groups having finite index
in the vertex groups.
\end{thm}
\begin{proof}
If $G$ is the fundamental group of a finite
graph of groups then $G$ acts on the associated Bass - Serre tree 
with vertex stabilisers that are conjugate in $G$ to some vertex group.
Moreover
this will be a locally finite tree if all edge groups have finite
index in the corresponding vertex groups. The action will have a global
fixed point if and only this decomposition is trivial.

Conversely suppose that $G$ acts on some locally finite tree $T$. Then
finite generation of $G$ means that we can take an invariant subtree
$T_0$ where the quotient $G\backslash T_0$ is a finite graph (with $T_0$
also locally finite), thus giving rise to an expression 
of $G$ as the fundamental group of a finite graph of groups where all 
edge groups have finite index in the vertex groups. 
Again this action will also
have a global fixed point if and only if the graph of groups
decomposition of $G$ is trivial. If further we know that this action is
faithful then
take a vertex $v$ and restrict the action of $G$ on $T$
to the vertex stabiliser $G_v$, which will also be a faithful action
on $T$
and with $v$ as a global fixed point. As $T$ is locally finite, Lemma 
\ref{resf} tells us that $G_v$ is residually finite, hence so are all
vertex groups.

Thus now we assume that $G$ has such a graph
of groups decomposition and we take the action of $G$ on the associated
Bass - Serre tree $T$. We
need to ensure a faithful action, which will be achieved
by adding subtrees at each vertex of $T$ so that point stabilisers act
faithfully, then we extend the action of $G$ equivariantly.

To do this, take $v_1,\ldots ,v_r$ to be representatives in the tree $T$
of the vertices in the finite quotient graph $G\backslash T$ and let
$H_1,\ldots ,H_r\leq G$ be the corresponding vertex stabilisers. We will
form a new tree $\overline{T}$ (which will still be locally finite) in which
$T$ embeds, by first adding a rooted tree $R_i$ to each vertex $v_i$
with $v_i$ as the root vertex. Each tree $R_i$ is obtained by applying
Lemma \ref{resf} to $H_i$, which is isomorphic to some vertex group
and hence residually finite, hence we have a natural action of $H_i$ on
$R_i$ with the root vertex $v_i$ fixed by $H_i$. 

Having done this, we now place a rooted tree at every other vertex $v$ of $T$.
We call this subtree $R_i^{(v)}$, where $v\in\mbox{Orb}(v_i)$ and we make it
naturally isomorphic to the subtree $R_i=R_i^{(v_i)}$ at vertex $v_i$.

We must now define the action of $G$ on the new tree $\overline{T}$, which
is the same as before on the subtree $T$. For each $1\leq i\leq r$
we have that the left cosets of $H_i$ in $G$ are in bijection with the
points in $\mbox{Orb}(v_i)$, so we can 
take a left transversal (infinite in general) of $H_i$ in $G$ which we
denote as $\{g_i^{(v)}: v\in\mbox{Orb}(v_i)\}$ and where we have
$g_i^{(v)}(v_i)=v$. Let us first take an
element $h_v$ in the stabiliser $H_v$ of some vertex $v$ which is in the
$G$-orbit of $v_i$. We define the action of $h_v$ on the subtree $R_i^{(v)}$
as follows: note that $(g_i^{(v)})^{-1}h_vg_i^{(v)}$ is in the stabiliser
$H_{v_i}$ which already has an action on the rooted tree $R_i$. Thus for
the vertex $r_i^{(v)}\in R_i^{(v)}$,
we define $h_v(r_i^{(v)})=s_i^{(v)}$, where we have 
$(g_i^{(v)})^{-1}h_vg_i^{(v)}(r_i)=s_i\in R_i$ and where $r_i^{(v)},s_i^{(v)}$ are
the vertices equivalent to $r_i,s_i$ under the natural isomorphism between
$R_i$ and $R_i^{(v)}$.

Now for an arbitrary element $\gamma\in G$ and a ``new'' vertex $r_i^{(v)}$
in the rooted subtree $R_i^{(v)}$ based at $v\in\mbox{Orb}(v_i)$, we will
let $w$ be the vertex $\gamma(v)$ and (uniquely) write 
$\gamma g_i^{(v)}=g_i^{(w)}h_{v_i}$. Setting 
$h_{v_i}=(g_i^{(v)})^{-1}h_vg_i^{(v)}$ which is in $H_i$, we define
$\gamma(r_i^{(v)})=s_i^{(w)}$ where $h_{v_i}(r_i)=s_i$ and the vertices
$r_i,r_i^{(v)}$ are the equivalent vertices in the subtrees 
$R_i,R_i^{(v)}$ and $s_i,s_i^{(w)}$ are the equivalents in $R_i$ and
$R_i^{(w)}$.

To show this is an action on the enlarged tree $\overline{T}$, the identity
still acts trivially on $\overline{T}$ so consider $\gamma,\delta\in G$
and a vertex $v$ in the original tree $T$, along with a vertex
$r_i^{(v)}$ in the subtree $R_i^{(v)}$ rooted at $v$ and where 
$v\in\mbox{Orb}(v_i)$. Let us also set $\delta(v)=w$ and $\gamma(w)=x$.
We then obtain unique elements $l_{v_i},h_{v_i}\in H_i$ where
$\delta g_i^{(v)}=g_i^{(w)}l_{v_i}$ and $\gamma g_i^{(w)}=g_i^{(x)}h_{v_i}$,
so that $\gamma\delta g_i^{(v)}=g_i^{(x)}h_{v_i}l_{v_i}$ for the unique
element $m_{v_i}=h_{v_i}l_{v_i}\in H_i$.

Suppose then that $\delta(r_i^{(v)})=s_i^{(w)}$ where $s_i\in R_i$ is the
image under $l_{v_i}$ of the point $r_i\in R_i$ equivalent to $r_i^{(v)}$,
and $s_i^{(w)}$ is equivalent to $s_i$. In the same way, set
$\gamma(s_i^{(w)})=t_i^{(x)}$ where $h_{v_i}(s_i)=t_i$ back in the subtree
$R_i$, so that we have $\gamma(\delta(r_i^{(v)}))=t_i^{(x)}$. But for the
product element $\gamma\cdot\delta\in G$, we have 
$\gamma\cdot\delta(v)=x$ from the action of $G$ on $T$. This means that
$\gamma\cdot\delta(r_i^{(v)})$ is obtained by writing
$\gamma\cdot\delta g_i^{(v)}$ as $g_i^{(x)}n_{v_i}$ for a unique
element $n_{v_i}\in H_{v_i}$ and then we have 
$\gamma\cdot\delta(r_i^{(v)})=u_i^{(x)}$ in the usual way, where
$n_{v_i}(r_i)=u_i\in R_i$.

So we have in $G$ that $\gamma\delta g_i^{(v)}=g_i^{(x)}m_{v_i}$
and $\gamma\cdot\delta g_i^{(v)}=g_i^{(x)}n_{v_i}$, thus 
$n_{v_i}=m_{v_i}=h_{v_i}l_{v_i}$ and hence $u_i=n_{v_i}(r_i)=h_{v_i}(s_i)=t_i$.
This gives us $u_i^{(x)}=t_i^{(x)}$ and
$\gamma(\delta(r_i^{(v)}))=(\gamma\cdot\delta)(r_i^{(v)})$, so we have our 
action.  
Moreover for any non identity element $g\in G$, if it fixes a vertex $v$
of $T$ then it is conjugate to an element in $H_i$ for the appropriate $i$.
But this conjugate acts faithfully on the rooted tree $R_i$ by definition
of the action on $R_i$ from Lemma \ref{resf}, thus overall the action
of $G$ on $\overline{T}$ is faithful.
\end{proof} 

We can also use this construction in other contexts, for instance as
mentioned just before Lemma \ref{resf}, we can turn an arbitrary action
of a group $G$ on a tree $T$ into a faithful action of $G$ on another tree
$T'$ by taking each vertex $v\in T$ in turn and adding a rooted star graph
at $v$ on which the stabiliser $G_v$ acts faithfully (making sure as in the
above proof that the action of the stabilisers on each of these rooted
trees is defined equivariantly for vertices in the same orbit under $G$).
For instance we have:

\begin{co} If a countable group $G$ acts faithfully with a global fixed point
on some locally finite tree $T$ and it also acts without a global fixed
point on some other locally finite tree $T'$ then there exists a
faithful action of $G$ on a locally finite tree without a global fixed point.
\end{co}
\begin{proof} We know $G$ will be residually finite from its action on
$T$ and therefore in its action on $T'$, every vertex stabiliser will
be residually finite. We can then add subtrees at every vertex to
ensure the action of $G$ is faithful, exactly as in the proof of
Theorem \ref{fthtre} except that there may be infinitely many
vertices in the quotient graph.
\end{proof}

We illustrate with some examples:
\begin{ex}
\end{ex}
(1) An example of a finitely generated group $G$ with no faithful action
on any locally finite tree would be if $G$ has Serre's property FA and is
also not residually finite, by Lemma \ref{resf}. If further $G$ has no
proper finite index subgroups then every action is trivial. In particular
we have in \cite{cap} by Caprace and R\'emy 
even finitely presented groups $S$ with property (T), thus
with FA, but which are simple.

If we now consider the free product $S*S$ then this group has an action
on the corresponding Bass - Serre tree of infinite degree with no global
fixed point and which is faithful (as for a free product no edge is
fixed by a non trivial element) but still every action on a locally finite
tree is trivial (as each factor $S$ would fix a vertex and so act trivially).
We can even consider examples such as the amalgamation 
$(S\times C_d)*_S (S\times C_d)$ for $C_d$ the cyclic group of order $d$,
where this splitting gives rise to
an action on a locally finite tree without a global fixed point, but still
without a faithful action on any locally finite tree.
\\
\hfill\\
(2) In \cite{bh1} Bhattacharjee
showed the existence of amalgamated free products
$G=F_k*_{F_l}F_k$ with $F_l$ having finite index in the rank $k$ 
free groups $F_k$ on either side but where $G$ has no proper finite index 
subgroups. Thus any action of $G$ on a locally finite tree with a global
fixed point is trivial, but this splitting of $G$ gives us an action without
a global fixed point on a locally finite tree.
Although it was mentioned in \cite{bh2} that it was
not known whether this particular action of $G$ is faithful, one can use 
Theorem \ref{fthtre} to turn this into a faithful action on a locally finite
tree. Alternatively the existence of such
examples which are simple, by the deep methods of Burger - Mozes in \cite{bm}, 
would give a faithful action (because the kernel of the action would be a 
normal subgroup).\\
\hfill\\
(3) The Baumslag-Solitar group $BS(2,3)$, which is not
residually finite, has non trivial actions on locally finite trees with a 
global fixed point but none of them can be faithful. However it does have
faithful actions on locally finite trees without a global fixed point by 
using the splitting obtained from the HNN extension.
 
\section{Uniformly bounded and regular trees}

A locally finite tree is said to be uniformly bounded if there is a finite 
upper bound for the degree of the vertices and $d$-regular (also
called homogeneous) if every vertex has degree $d$. We are interested
in how the results of the last section can be strengthened when our
tree $T$ is uniformly bounded or even regular. For this,
it will also be useful to consider a slightly different tree: let $R_d$ be
the regular rooted tree of degree $d$, so that there is a root vertex
of degree $d$ but every other vertex has degree $d+1$ (thus any automorphism
of $R_d$ will have to fix the root vertex).

If a locally finite tree $T$ embeds in another locally finite tree $S$
then this does not mean in general that $Aut(T)\leq Aut(S)$. However
if $T$ is uniformly bounded with all vertex degrees at most $d$ say
then $T$ certainly embeds in the regular tree $T_d$: for each vertex $v$
in $T$ of degree $1\leq i<d$, we add $d-i$ edges to $v$ and then at
each of the new $d-i$ vertices we place a $(d-1)$ degree regular rooted
tree, so that now every vertex has degree $d$. Moreover we will still have
$Aut(T)\leq Aut(T_d)$ by extending automorphisms as in Theorem \ref{fthtre}.
Indeed given vertices $v,w\in T$ which are equivalent under some automorphism
$\alpha$ of $T$, they will both have the same degree $i$. Thus if $i<d$ then
we extend the action of $\alpha$ to the $d-i$ degree regular rooted trees
which were added at $v$ by mapping them to those rooted trees at $w$
using a canonical isomorphism (the identity  if $v=w$).
Similarly if $R$ is a rooted tree with the root having degree at most $d$
and all other vertices having degree at most $d+1$ then $R$ embeds in the
regular rooted tree $R_d$ with the root sent to the root. Moreover the
same argument as before gives us that $Aut(R)$ is a subgroup of
$Aut(R_d)$ (where in the rooted case we take $Aut(R)$ to be the automorphisms
of $R$ which fix the root vertex).

Thus if we are given a group $G$ then asking if $G$ acts faithfully on some
uniformly bounded tree is the same as asking whether $G$ acts faithfully
on a regular tree of some finite valency. Consequently we should first
consider an equivalent version of Lemma \ref{resf}, where we give a
characterisation of which groups act faithfully on the regular tree $T_d$
with a global fixed point.
We will also consider results for the regular rooted tree $R_d$ as well.

\begin{defn}
Given an integer $n\geq 2$,
we say that a group $G$ has an {\bf index $n$-chain}
if there is a decreasing sequence of finite index subgroups
$G=G_0>G_1>G_2$ with $\cap_{i=0}^\infty G_i=\{id\}$ and such that 
for $i\geq 0$ the index $[G_i:G_{i+1}]$ is at most $n$.
\end{defn}
\begin{thm} \label{resfd} For $d\geq 2$,
a group $G$ acts faithfully on the $d$-regular rooted tree $R_d$
with a global fixed point if and only if $G$ has an index $d$-chain.
For the $d+1$-regular tree $T_{d+1}$, a group $G$ acts faithfully with a global
fixed point if and only if $G$ has an index $d$-chain, except that
in this case we allow the index $[G_0:G_1]$ to be at most $d+1$, rather than
$d$.
\end{thm}
\begin{proof}
First suppose existence of the chain then, on starting with a root vertex
$v_0$ for $G_0$ and creating the coset tree $T$ as in Lemma \ref{resf}, we
have that the degree of $v_0$ is the index $[G_0:G_1]$ which is at most $d$
in the rooted case and $d+1$ for $T_d$. As for the other vertices,
a vertex $v$ in level $l\geq 1$ of the tree will be joined to one vertex in
the previous level and $[G_l:G_{l+1}]$ vertices in level $l+1$,
thus the tree $T$ has degree at most $d+1$ in both cases.  
Moreover $G$ acts faithfully
on $T$ because any element fixing all vertices would lie in
$\cap_{i=0}^\infty  G_i$. Thus without loss of generality we can assume
that $T$ is either the regular rooted tree $R_d$ or the regular tree $T_{d+1}$.

Now suppose for $d\geq 2$ that our group $G$ acts faithfully with global 
fixed point
$v_0$ on either the $d$-regular rooted tree $R_d$ or the $d+1$-regular tree
$T_{d+1}$. We will label the vertices in either tree
as $v_{l,i}$ where $l$ is the level of the vertex from the root $v_0=v_{0,1}$
at level 0 and $i$ runs from 1 to either $d^l$ or
$(d+1)d^{l-1}$ according to some natural ordering of 
the level $l$ vertices which lists together those vertices with the same
parent in level $l-1$. We denote the stabiliser in $G$ of $v_{l,i}$ by
$H_{l,i}$, so that $G=H_{0,1}$.

We now consider the sequence of intersected stabilisers
\[G_1=H_{1,1}, G_2=H_{1,1}\cap H_{1,2}, \ldots, G_k=H_{1,1}\cap 
H_{1,2}\cap\ldots \cap H_{1,k}=G_{k-1}\cap H_{1,k}\]
(where $k=d$ for the rooted case and $d+1$ otherwise)
and when we reach the next level we continue to intersect stabilisers
in order so that $G_{k+1}=G_k\cap H_{2,1}$ and so on.
Now let us consider the index of $G_n$ in $G_{n-1}$ for some arbitrary
$n$. As $G_n=G_{n-1}\cap H_{l,i}$ for some $l,i$, the index $[G_{n-1}:G_n]$ is 
given by the size of the
orbit of the vertex $v_{l,i}$ under the action of the group
$G_{n-1}$. As this will fix the parent vertex of $v_{l,i}$, if the level 
$l\geq 2$ 
then we have that the size of the orbit of the vertex
$v_{l,i}$ is at most $d$. For level $l=1$ it will again be at most $d$ for
$R_d$ but at most $d+1$ for $T_{d+1}$.  

Thus on continuing this process we obtain a descending sequence 
$G=G_0\geq G_1\geq G_2\geq \ldots$ of finite index subgroups, possibly with
repeated subgroups where we have equality, and such that the
intersection $\cap_{n=0}^\infty G_n$ is trivial because it is contained in every 
vertex stabiliser and the action of $G$ is faithful. Moreover for the
rooted tree $R_d$, all indices $[G_{n-1}:G_n]$ 
are at most $d$ so we have our index $d$-chain in this case on removing
all repeats from our sequence.

As for the case of $T_{d+1}$, all indices $[G_{n-1},G_n]$ are at most $d$,
with the exception of the indices obtained when intersecting
level 1 stabilisers which could be at most $d+1$. However
we will now show that, on removing all repeats from our descending sequence
(and possibly allowing a slight reordering),
only $[G_0:G_1]$ can equal $d+1$ which will complete the proof for $T_{d+1}$.  

First consider the $d+1$ level 1 vertices $v_{1,1},\ldots ,v_{1,{d+1}}$ 
which must be permuted
amongst themselves by any action of $G$ on $T_{d+1}$, thus 
$|Orb(v_{1,i})|\leq d+1$.
By Orbit - Stabiliser, we have $|Orb(v_{1,1})|=[G:H_{1,1}]\leq d+1$ so either
there is some $i$ with $1<[G:H_{1,i}]\leq d+1$ or
$G$ fixes all level 1 vertices.
In this latter case, we have $G_0=G_1=\ldots=G_{d+1}$ before removing repeats 
and so our first 
proper subgroup in this sequence comes when intersecting a stabiliser of a 
vertex from level $l\geq 2$, thus this index is at most $d$ as above.

In the former case we relabel $v_1,\ldots ,v_{d+1}$ (and the other vertices
$v_{l,i}$ correspondingly) so that the new vertex $v_{1,1}$ is any previous
vertex $v_{1,i}$ where we had $1<[G:H_{1,i}]$. Thus now  $G_1$ is the stabiliser
$H_{1,1}$ of the vertex $v_{1,1}$ and is a proper subgroup with
$[G:G_1]\leq d+1$. Then
for $2\leq i\leq d+1$ (in the new numbering)
we have that $[G_{i-1}:G_i]=[G_{i-1}:G_{i-1}\cap H_{1,i}]$
so is equal to the orbit of the vertex $v_{1,i}$ in the group $G_{i-1}$. But
as $i\geq 2$, some level 1 vertices will be fixed by $G_{i-1}$, thus
this orbit has size at most $d$.
\end{proof}

Note: suppose that we have a group $G$ acting on the regular tree $T_d$ with 
a global fixed point $v_0$ but that $G$ does not have a proper subgroup of 
index
at most $d$. Then by the argument above $G$ will fix all $d$ level 1 vertices
but we can then iteratively apply this argument on each subtree to find that 
the action of
$G$ is not just unfaithful but trivial. This also holds for $G$ acting on
the $d$-regular rooted tree. 

We can now obtain an equivalent version of Theorem \ref{fthtre} which gives
a complete characterisation of finitely generated groups embedding into the 
automorphism group of some uniformly bounded or locally finite tree.

\begin{co} \label{unib}
A finitely generated group $G$ embeds in the automorphism group
of some uniformly bounded tree if and only if it embeds in the automorphism
group of some regular tree if and only if it can be expressed as a finite 
graph of groups (possibly trivial)
with edge groups having finite index in the vertex groups
and where all vertex groups possess an index $n$-chain for some $n$.
\end{co}
\begin{proof}
If there is such a decomposition then we have our action on the Bass - Serre
tree as before. The degree at each vertex $v$ is the sum of the indices of the 
inclusion of each edge group in this vertex group, over all edges incident
at $v$, thus is finite. As there are only finitely many vertex groups
in the decomposition, this tree is uniformly bounded. To ensure the action
is faithful, we combine Theorem \ref{resfd} with the proof of 
Theorem \ref{fthtre} to add (rooted) subtrees at each vertex so that all 
vertex stabilisers embed. This increases the degree at each of these vertices by
$n$, but again as there are only finitely many vertex groups  and each has
its own index $n$-chain, there are only finitely many values of $n$
occurring.

Now suppose $G$ is finitely generated and a subgroup of $Aut(T)$ for
some uniformly bounded tree $T$. 
We can assume that $G$ acts without edge inversions by barycentric subdivision,
in which case the tree is still uniformly bounded.
Again we use finite generation to obtain an invariant subtree $T_0$ with
$G\backslash T_0$ finite and hence a finite graph of groups decomposition with
edge groups finite index in the vertex groups (which is a trivial
decomposition if and only if $G$ acts with a global fixed point). Finally
we regard each of the vertex groups as the stabiliser of some vertex of
this uniformly bounded tree $T_0$, so that we can apply Theorem \ref{resfd}
to $T_0$, or at least to any $d$-regular tree containing $T_0$, which shows
that each vertex group has an index $n$-chain for some $n$.
\end{proof}

We also obtain
\begin{co}
Suppose that a finitely generated group $G$ acts (not necessarily
faithfully) on some locally 
finite tree without a global fixed point, but also acts faithfully on some
regular tree $T_d$ with a global fixed point. Then $G$ acts faithfully
on some regular tree $T_{d'}$ without a global fixed point.
\end{co}
\begin{proof}  
The first action gives rise to an non trivial
expression of $G$ as a finite graph of groups with edge groups having
finite index in the vertex groups. Now $G$ will have an index $n$-chain
for some $n$ from the second action by Theorem \ref{resfd}. 
But this condition is preserved by
arbitrary subgroups, thus all vertex subgroups in this decomposition
have an index $n$-chain and so Corollary \ref{unib} applies.
\end{proof}
   
We illustrate with some more examples.
\begin{ex}
\end{ex}
(1) To give an example of a finitely generated group $G$ that acts
faithfully on some locally finite tree but with no faithful
action on a uniformly bounded tree, we will need first need $G$ to
be residually finite. Suppose further that there are infinitely many
primes $p_i$ for which there is some element $g_i\in G$ with
order a multiple of $p_i$, thus some power of $g_i$ has exact order $p_i$.
Then for any tree $T$ with vertex degree bounded above by $d$ and any action
of $G$ on $T$, take an element $g$ with prime order
$p_i>d$. Then $g$ fixes a vertex, so fixes every vertex of $T$. 
Such groups were constructed in
\cite{olos} Section 3; see Lemma 3.3 of that paper. If it is required that
the faithful action does not have a global fixed point then we can take a 
proper finite index subgroup $H$ of $G$ and replace $G$ with $G*_HG$.\\
\hfill\\ 
(2) Let $G$ be an infinite, finitely generated, residually finite $p$-group 
(for $p\geq 3$ a
prime) which acts faithfully on the rooted regular tree $R_p$ (say the
the Gupta - Sidki groups). Any action of $G$ on a tree must have a global
fixed point, as $G$ is torsion and finitely generated so cannot split. 
Moreover, as $G$ has no proper
subgroup of index less than $p$ (and the same for any subgroup of $G$),
any action
of $G$ on the $p$-regular tree $T_p$ is either trivial or has cyclic
image $C_p$. This is because on taking $H\leq G$ to be 
the stabiliser of a level 1 vertex, we have by Orbit - Stabiliser
that either $H=G$ or $[G:H]=p$ and $H$ fixes all $p$ level 1 vertices $v$
in both cases. Thus for any such $v$,
we have by the note after Theorem \ref{resfd} that $H$ acts trivially on each 
regular rooted subtree $R_{p-1}$ with root vertex 
$v$ as $H$ has no proper subgroups of index less than $p$. \\
\hfill\\
(3) For any prime $p$ the free group $F_k$ is residually finite-$p$.
However a residually finite-$p$ group $G$ has an index $n$-chain whenever
$p\leq n$. To see this, suppose we have $G=N_0>N_1>N_2>\ldots $ where
each $N_i$ is normal in $G$ with index a power of $p$ and with 
$\cap_{i=0}^\infty N_i=\{e\}$. Then $P=G/N_1$ is a finite $p$-group
and so has a subgroup of index $p$ which can be pulled back to
one for $G$. We can now replace $G$ by this subgroup and continue.
Thus $F_k$ acts faithfully with a global fixed point on the regular
tree $T_d$ for $d\geq 3$ and the regular rooted tree $R_d$ for $d\geq 2$.

We also note that a finitely generated linear group $G$ is virtually 
residually finite-$p$ for all but finitely many primes $p$ in characteristic
0 and for $p$ itself in characteristic $p$.
Thus if $G$ has a subgroup $H$ of index $m$ which is residually 
finite-$p_0$ then $G$ has an index $n$ chain for $n=\max(m,p_0)$. In
particular any finitely generated linear group acts faithfully on some
regular tree with a global fixed point.\\
\hfill\\
(4) Let $G$ be the free product $C_p*C_p=\langle g,h|g^p=h^p=e\rangle$ 
for $p$ prime, which
does not have any proper subgroups $H$ of index less than $p$ (else 
$g^i,h^j\in H$ for $1\leq i,j <p$ hence $g,h\in H$).
First suppose
that $d<p$, in which case
any action of $G$ on the regular tree $T_d$ with a global fixed point
or the rooted regular tree $R_d$ must be trivial. Moreover any
action of $G$ on $T_d$ without a global fixed point must be trivial too
because the finite subgroup $\langle g\rangle$ must act with a fixed
point, thus $g$ will act trivially on $T_d$ but $h$ will too.
 
For $d=p$ we have the faithful action of $G=C_p*C_p$ on $T_p$
obtained from this splitting. We also have a faithful action
of $G$ on $T_p$ with a global fixed point (and one on $R_p$ too):
Consider the homomorphism $\theta:C_p*C_p\rightarrow C_p=\{0,1,\ldots ,p-1\}$ 
given by sending both $g$ and $h$ to 1. As any torsion element of $G$ is
conjugate to a power of $g$ or of $h$, we have that $K=\mbox{ker}(\theta)$
is torsion free and thus is a finitely generated free group so that
$G=K\rtimes C_p$. We now let $C_p$ rotate around the global fixed point
$v_{0,0}$ and $K$ act faithfully on the rooted tree based at the vertex 
$v_{1,1}$ as in (3), extending the action to the rooted trees based
at $v_{1,2},\ldots ,v_{1,p}$ on conjugating $K$ by $g,\ldots ,g^{p-1}$
respectively.  Hence $C_p*C_p$ will also act faithfully, both
with and without a global fixed point, on $T_d$ and $R_d$ for $d>p$.\\
\hfill\\

It was originally 
shown in \cite{z} that the automorphism group $Aut(T_d)$ is
not abstractly isomorphic to $Aut(T_{d'})$ if $d\neq d'$. Other proofs 
were given in \cite{mol} and \cite{bslb}. Indeed the latter paper
generalises this widely as it gives conditions on locally finite
trees $T,T'$ such that any abstract group isomorphism
$\theta:Aut(T)\rightarrow Aut(T')$ is induced by a tree isomorphism
from $T$ to $T'$. However we can give here a very quick
and basic proof.

\begin{thm} If $T_d$ and $T_{d'}$ are the regular trees of degree $d,d'$
respectively then $Aut(T_d)$ is not isomorphic to $Aut(T_{d'})$ as an
abstract group.
\end{thm}
\begin{proof} We can suppose that $3\leq d<d'$, whereupon
we simply find a finite 
subgroup of $Aut(T_{d'})$ which cannot act faithfully on $T_d$, thus 
cannot be contained in $Aut(T_d)$.
If $d'\geq 5$ then $H=Alt(d')$ clearly acts faithfully on $T_{d'}$ with
a global fixed point, but $H$ has no proper subgroups of index less
than $d'$. Now suppose $H$ acts on $T_d$. As it is finite, it must
have a global fixed point but then Theorem \ref{resfd} tells us that
this cannot be a faithful action.  

This just leaves $d=3$ and $d'=4$ where we can take $H=C_3\times C_3$
which acts faithfully on $T_4$ (by fixing one level 1 vertex and the
subtree below it) but not on $T_3$ by Theorem 3.2.
\end{proof}

The paper \cite{bslb}
also proves the same result for biregular trees $T_{d_1,d_2}$
and $T_{d_{1'},d_{2'}}$ where $d_1\geq d_2\geq 3,d_{1'}\geq d_{2'}\geq 3$ and
${d_1,d_2}\neq {d_1',d_2'}$. Note that our proof works immediately
in this case provided the larger degrees are distinct, say $d_{1'}>d_1$.
In fact it can be adapted for the $d_{1'}=d_1$ case too by considering
finite sections (quotients of subgroups) rather than just finite
subgroups. However the arguments are somewhat ad hoc, so instead we
will finish this section by looking at the equivalent results for
automorphism groups of regular rooted trees. Exactly the same
argument with alternating groups works in this case
to show that $Aut(R_d)\not\cong Aut(R_{d'})$ for $d<d'\geq 5$,
and we can use the group $C_3$ to distinguish between $Aut(R_d)$ and
$Aut(R_{d'})$ when $d=2$ and $d'=3$ or $4$. This only leaves $Aut(R_3)$
and $Aut(R_4)$. However here something
rather strange happens which does not seem to have been noticed before.
Having failed to find a subgroup of $Aut(R_4)$ which does not embed
in $Aut(R_3)$, we noticed the following basic result.

\begin{prop} \label{threef}
If a group $G$ has a proper subgroup $H$ of index 4 then $G$ has a proper
subgroup of index 2 or 3. Indeed suppose that $K$ is the core of $H$ in $G$.
We then have a finite sequence of subgroups $G=G_0>G_1>\ldots >G_k=K$
with the index $[G_{i-1}:G_i]$ at most 3 for $1\leq i\leq k$.
\end{prop}
\begin{proof}
If $[G:H]=4$ and we consider the usual action $\rho$ of $G$ on the left cosets
of $H$ with kernel $K$
then the image $\rho(G)$ is a transitive subgroup of $Sym(4)$, thus it
has order $4,8,12$ or 24 and $H$ is the stabiliser of a point. In the first
case we can see that $\rho(G)$ has a subgroup of index 2 which itself
contains $\rho(H)$ of index 2, so the same is true in $G$ by
pullback and here we have $H=K$. If $\rho(G)$ has order 8 then we can think of 
this as the dihedral group with $\rho(H)$ a reflection, and we can do the
same as before with $K$ having index 2 in $H$.

If $\rho(G)$ has order 12 and therefore is $Alt(4)$ then $\rho(H)$ must
be a copy of $C_3$. Although there is no subgroup between $Alt(4)$
and $C_3$, we can descend from $Alt(4)$ to $C_2\times C_2$ to $C_2$ to
the identity in steps of 2 and 3, then pull back as before. Similarly
if $\rho(G)=Sym(4)$ and hence $\rho(H)$ is $Sym(3)$ then we can go from 
$Sym(4)$ to $Alt(4)$ and then as above, resulting in steps of size 2,3,2,2.
\end{proof}

\begin{thm}
For $R_3$ and $R_4$ the regular rooted trees of degree 3 and 4 (which refers
to the degree of the root vertex), we have that $Aut(R_3)\leq Aut(R_4)$
and $Aut(R_4)\leq Aut(R_3)$ as abstract groups.
\end{thm}
\begin{proof}
Only the second statement needs proof, so let us take $G=Aut(R_4)$ which has
an index 4-chain by Theorem \ref{resfd}. By the same theorem (which does
not need the group in question to be countable) we are done if we can
show that $G$ has an index 3-chain.

Suppose first that we are in a situation where
$G=G_0>G_1>G_2>\ldots$ with $\cap G_i=\{I\}$ and $[G_0:G_1]=4$ but
$[G_{i-1}:G_i]\leq 3$ for all $i\geq 2$. By Proposition \ref{threef} we will
have a small integer $k$ and subgroups $H_1>\ldots >H_k$ of $G$ such that
$G:=H_0>H_1>\ldots >H_k=K_1$ for $K_1$ the core of $G_1$ in $G$ and where
$[H_{i-1}:H_i]\leq 3$ for $1\leq i\leq k$. Now $K_1=K_1\cap G_1\geq K_1\cap G_2
\ldots$ is an index 3-chain for $K_1$ (at least on removing repeats), so is
also one for $G$ on adding $H_1,\ldots ,H_k$ to the start.

In the general case where we are given an index 4-chain
$G=G_0>G_1>G_2>\ldots$, we perform the above step if $[G_0:G_1]=4$
but do nothing if $[G_0:G_1]<4$. This results in a new index 4-chain
which starts at $G$ and is actually an index 3-chain as far down as
$K_1=K_1\cap G_1$. Assuming now that $[K_1:K_1\cap G_2]=4$ (otherwise
we do nothing), we perform the same process as above 
for $K_2$ the core of $K_1\cap G_2$ in $K_1$. This results in an index 4-chain
for $G$ which is now an index 3-chain as far down as $K_2$, which is
a subgroup of $G_2$. We continue at each step by taking $K_{j+1}$ to
be the core of $K_j\cap G_{j+1}$ in $K_j$, removing all steps of index 4
from our original chain. Moreover we have $K_j\leq G_j$ and $\cap G_j=\{I\}$
implies the same for $\cap K_j$, so we have an index 3-chain for $G$.
\end{proof}

\section{Proper actions on products of trees}

We now consider actions of groups on trees (first we consider arbitrary
trees but then we specialise to the locally finite case)
which are free and/or proper, as well as such actions on a finite
product of trees.
A classical result (see \cite{ser} I.3) is that a group $G$ acts
freely on a tree $T$ if and only if $G$ is a free group. Note that this
statement is completely general: it does not require $G$ to be finitely 
generated, nor is there any restriction on $T$ such as being locally
finite. We might now hope that $G$ acts properly on a tree if
and only if $G$ is virtually free but unlike free actions, this 
is not true in full generality and (at least in the case of arbitrary trees) 
requires explanation
of which of the several definitions of proper action is being used.

For a countable group $G$ acting on a metric space $X$ by isometries, a
standard definition of a {\bf proper action} is that for every compact
subset $C$ of $X$, the set of elements $\{g\in G: g(C)\cap C\neq\emptyset\}$ 
is finite. However we also have the stronger condition of being {\bf 
metrically proper} where we replace compact with bounded, or equivalently 
closed and bounded. Thus if $X$ is a proper metric space, meaning that all 
closed balls are compact, such as a locally finite tree then these two notions
are equivalent, but for example if an infinite group $G$ is given the discrete
metric then the isometric action of $G$ on itself by left multiplication
is proper but not metrically proper.

A similar definition, which we will refer to as {\bf properly 
discontinuous} though some authors call this proper, is that
for every point $x\in X$ there exists an open neighbourhood $U$ of
$x$ such that the set $\{g\in G: U\cap g(U)\neq\emptyset\}$ is finite.

Another issue is that if $G$ is a subgroup of some topological group
$\mathcal G$ then proper is sometimes used in place of saying that $G$ is 
discrete (namely inherits the discrete topology from $\mathcal G$). Now 
we can turn the isometry group $Isom(X)$
into a topological group by giving it the compact open topology, so that we 
have a sub-basis of open sets $O(K,U)$ where for $K$ compact and
$U$ open subsets of $X$, we define $O(K,U)=\{f\in Isom(X): f(K)\subseteq U\}$.
(In fact some care is needed here because if we replace $Isom(X)$ with the
group $Homeo(X)$ and put the compact open topology on it then this need
not be a topological group in general, even for quite nice metrizable
spaces. However this is true for isometries
of an arbitrary metric space, see for instance
Proposition 5.1.3 of \cite{dassu} or Section 5.B of \cite{dlhc}
where
the idea is that the compact open topology is the topology of uniform
convergence on compact subsets, whereas the topology of pointwise convergence
does define a topological group. However these notions are the same for
isometries.)

Let us now consider actions of groups on products of trees.
Suppose that $P$ is a finite product $T_1\times\ldots \times T_n$ of
trees where the trees need not be isomorphic. We put a
cube complex structure on $P$ in the obvious way and we equip $P$ with
the $\ell_2$ product metric, using the path metric on each factor, so
that $P$ is a finite dimensional CAT(0) cube complex. If further all
trees are locally finite then $P$ is a locally finite cube complex and
also a proper metric space. We then consider groups $G$ acting
on $P$, so that there is a homomorphism from $G$ to $Aut(P)$ where
the cube complex structure is preserved (and so $G$ will be acting by
isometries of $P$).
This means that $G$ could act
by permuting the factor trees (although it can obviously only permute
isomorphic trees), though there will always be a finite index subgroup
$H$ of $G$, written $H\leq_f G$, which preserves all factors. Then
we can consider $H$, or at least the image of $H$ in
$Aut(P)$, as a subgroup of
$Aut(T_1)\times\ldots \times Aut(T_n)$, in which case we say that
$H$ acts {\bf preserving factors}. In this situation, we obtain for each
$1\leq i\leq n$ an action of $H$ on the tree $T_i$, given by the natural
projection $\pi_i:Aut(T_1)\times\ldots \times Aut(T_n)\rightarrow Aut(T_i)$.

As for the type of actions we will be interested in, we can also
look at the different definitions of proper actions and relate these 
as follows:
\begin{prop} \label{eqpr}
Let $P=T_1\times\ldots\times T_n$ be a finite product of trees (where each
factor tree is equipped with its path metric) and let 
$G$ be a group acting on $P$ by automorphisms (and hence isometries). 
Consider the following statements:\\
(i) $G$ acts metrically properly on $P$.\\
(ii) $G$ acts properly on $P$.\\
(iii) For all vertices $v$ in $P$, the stabiliser $G_v$ is a finite group.\\
(iv) $G$ acts properly discontinuously on $P$.\\
(v) The homomorphism $h:G\rightarrow Aut(P)$ given by the action
has finite kernel
and $h(G)$ is a discrete subgroup of $Aut(P)\leq Isom(P)$ 
with the compact open topology, or equivalently the topology of pointwise
convergence.\\
If the trees are not assumed to be locally finite then 
(i) strictly implies (ii) which is equivalent to (iii) and
(iv), which in turn strictly implies (v).
\end{prop}
\begin{proof}
In the general case,
it is clear that (i) implies (ii) as compact sets are closed and bounded,
and (ii) implies (iii) by considering the closed and bounded set 
$\{v\}$ for $v$ a vertex of $P$.

We now note that if $G$ acts by isometries on a metric 
space $X$ and $H\leq_f G$ then $H$ having any of properties (i) to (iv)
implies that $G$ will have the equivalent property too. 
This means we will be able to assume thereafter that
$G$ acts preserving factors. This is clear for (iii) and as for (i) and (ii),
suppose there were a compact/closed and bounded
subset $C$ of $X$ with $S=\{g\in G: g(C)\cap C\neq\emptyset\}$ infinite. Then
on taking a right coset decomposition 
$H\cup Hg_1\cup \ldots Hg_k$ we have that $D=C\cup g_1(C)\ldots \cup g_k(C)$
is compact/closed and bounded, with some $j$ for which $S\cap Hg_j$ is 
infinite, thus $\{h\in H:D\cap h(D)\neq\emptyset\}$ is infinite too.  

If $H$ has property (iv), suppose there is $x\in X$ such that for all 
open neighbourhoods $U$ of $x$, the set 
$S_U=\{g\in G: g(U)\cap U\neq\emptyset\}$ is infinite. If $S_U\cap G_x$
is infinite then so is $H_x$, thus we may assume (by shrinking $S_U$)
for a given $U$ that $S_U$ is infinite but that it contains no
element fixing $x$. We use this to obtain inductively
a sequence in $G$ of group elements 
$(g_i)$ where $g_i(x)\neq x$ are distinct points 
tending to $x$ as follows: for the open ball $B=B(x,d(x,g_i(x))/3)$, we will
have some $y\in B$ and an element $g_{i+1}$ not fixing $x$ such that 
$g_{i+1}(y)\in B$ too. Thus 
\[0<d(x,g_{i+1}(x))\leq d(x,g_{i+1}(y))+d(g_{i+1}(y),g_{i+1}(x))
\leq 2 d(x,g_i(x))/3.\]
Without loss of generality these elements $g_i$ all lie in the same
left coset of $H$, so there is $g\in G$ such that $g_i=gh_i$ for $h_i\in H$.
Then 
\[d(g_i^{-1}g_{i+1}x,x)\leq d(g_i^{-1}g_{i+1}x,g_i^{-1}x)+d(g_i^{-1}x,x)
=d(g_{i+1}x,x)+d(g_ix,x)\] 
so that $g_i^{-1}g_{i+1}(x)$ tends to $x$, but $g_i^{-1}g_{i+1}\in H$.
Moreover the sequence $g_i^{-1}g_{i+1}$ contains infinitely many
distinct elements, as otherwise $g_i^{-1}g_{i+1}(x)=x$ for all $i$ at least
some value $I$, which means that the distinct elements 
$g_I^{-1}g_{I+1},g_I^{-1}g_{I+2},\ldots $ all stabilise $x$.

We now prove that (iii) implies (ii). 
Suppose that the group $G$ acts on $P=T_1\times \ldots \times T_n$
preserving factors and we have a (non empty)
compact subset $C$ of $P$. Then the
projection $C_i$ of $C$ is a (non empty)
compact subset of $T_i$, thus we can enlarge
$C$ so that it is of the form $C_1\times\ldots\times C_n$. We will work
in each tree separately: for each $i$ consider the collection of open
sets consisting of all metric balls
$B(v^{(i)},1/2)$ over all vertices $v^{(i)}$ of $T_i$. This is almost an open
cover of $T_i$ but misses out the midpoint $m^{(i)}$ of each edge $e^{(i)}$,
thus on taking some small constant $r>0$ and adding the balls 
$B(m^{(i)},r)$ for every edge $e^{(i)}$ to our collection, we now have
an open cover of $T_i$. Thus compactness of $C_i$ implies that
there will be a finite set of vertices
$F^{(i)}=\{v^{(i)}_1,\ldots ,v^{(i)}_{k(i)}\}$ and of edges 
$E^{(i)}=\{e^{(i)}_1,\ldots ,e^{(i)}_{l(i)}\}$ in the tree $T_i$ such that
our subset $C_i$ is covered by
\[B(v^{(i)}_1,1/2)\cup\ldots\cup B(v^{(i)}_{k(i)},1/2)
\cup B(m^{(i)}_1,r)\cup\ldots\cup B(m^{(i)}_{l(i)},r).\]
We now enlarge this cover (but still keep it finite) by adding
to the set $F^{(i)}$ all endpoints of edges in $E^{(i)}$.  

Let us now take an element $\gamma$ in $Aut(T_i)$ where
$\gamma(C_i)\cap (C_i)\neq\emptyset$, so that there is a point
$c\in C_i$  with $\gamma(c)$ in $C_i$ too. If $c$ is in some $B(v^{(i)}_j,1/2)$
for $v^{(i)}_j\in F^{(i)}$ then the covering of $C_i$ means that
$\gamma(c)$ will either lie in some $B(v^{(i)}_{j'},1/2)$ 
for $v^{(i)}_{j'}$ also in $F^{(i)}$, or $\gamma(c)$ lies in some edge
in $E^{(i)}$. But as $c$ cannot be a midpoint and we enlarged the collection
of open 1/2-balls, we can assume that $\gamma(c)$ lies in 
the former set which implies that
$\gamma(v^{(i)}_j)=v^{(i)}_{j'}$. If however $c$ does not lie in one of these 
1/2-balls then it must be the midpoint of an edge in $E^{(i)}$ and the
same holds for $\gamma(c)$ too. Thus $\gamma$ will send the two endpoints of the
first edge to that of the second edge, and all of these endpoints are in
the finite set $F_i$. Thus to summarise, if $\gamma\in Aut(T_i)$ is such that
$C_i\cap \gamma(C_i)\neq\emptyset$     
then there will exist vertices
$v^{(i)}_g,w^{(i)}_g\in F^{(i)}$ with $\gamma(v^{(i)}_g)=w^{(i)}_g$.

We now go back to our group $G$: as we
are assuming that $G$ acts preserving factors, we can write any $g\in G$
as $g=(g_1,\ldots ,g_n)$ for each component $g_i$ an automorphism of $T_i$.
Let us consider the set $S=\{g\in G:g(C)\cap C\neq\emptyset\}$. For any
$g\in S$ we will have $g_i(C_i)\cap C_i\neq\emptyset$ for each $i$.
Thus if $S$ is infinite
then there exist a pair of vertices $v^{(1)},w^{(1)}$ in the finite set $F^{(1)}$
created above with the following property: infinitely
many elements $g\in S$ have their first component 
$g_1$ satisfying $g_1(v^{(1)})=w^{(1)}$. Throwing away all other elements
from $S$, we similarly
have vertices $v^{(2)}$ and $w^{(2)}$ in the set $F_2$ such that the
second component $g_2$ of each $g$ in the modified set $S$ sends $v^{(2)}$ to
$w^{(2)}$. Continuing in this way, we end up with infinitely many elements
$g=(g_1\ldots ,g_n)\in G$ with $g(v^{(1)},\ldots ,v^{(n)})=
(w^{(1)},\ldots ,w^{(n)})$, thus the vertex $(v^{(1)},\ldots ,v^{(n)})$ of $P$
is stabilised by infinitely many elements in $G$.    

To see that (iii) does not imply (i) in general, an easy example is to take
the action of an infinite rank free group on its Cayley graph under a free
generating set, which is a free action on a (non locally finite) tree but which 
is not metrically proper because any element of this infinite
generating set moves the identity vertex by 1.  

As for (iv), it is clear that this implies (iii). For the converse, we again
suppose that $G$ acts preserving factors and for any point 
${\bf p}=(p_1,\ldots ,p_n)\in P$ we take the open neighbourhood
$U=U_1\times \ldots \times U_n$ where $U_i$ is either the 1/2-ball
$B(p_i,1/2)$ if $p_i$ is a vertex of $T_i$, or the open edge $e_i$ in
$T_i$ that contains $p_i$. Now suppose that we have 
${\bf x}=(x_1,\ldots ,x_n)\in U$ and $g\in G$ where 
$g({\bf x})=(g_1(x_1),\ldots ,g_n(x_n))$is also in $U$. If $x_i$ is a vertex 
then we have $x_i=p_i=g_i(x_i)$ and if not then $g_i$ must send the edge
$e_i$ to itself. Then on taking for each $i$ either the vertex $x_i$ or
the two vertices at either end of $e_i$, we end up with at most $2^n$
vertices of $P$ which are permuted amongst themselves by the action of $g$.
Thus if there are infinitely many such elements then there exists infinitely
many elements in the stabiliser of any one of these vertices.

Now suppose that (v) does not hold. We first assume that $G$ embeds in
$Aut(P)$ via $h$, but not discretely. Now on taking any vertex
${\bf v}=(v_1,\ldots ,v_n)$ of $P$ and the open set
\[U_{\bf v}=B(v_1,1/2)\times \ldots \times B(v_n,1/2),\]
we have that the identity map lies in the sub-basic open set
$O(\{ {\bf v}\},U_{\bf v})$ but if $g\in Aut(P)$ also lies in
$O(\{ {\bf v}\},U_{\bf v})$ then, as it sends vertices to vertices, we
must have $g({\bf v})={\bf v}$. Thus if 
$O(\{ {\bf v}\},U_{\bf v})$ contains infinitely many elements of $G$ then
$G_{\bf v}$ is infinite but if it contains only finitely many elements
$id,g(1),\ldots ,g(k)$ then the open set
\[O(\{ {\bf v}(1)\},U_{\bf v}(1))\cap\ldots\cap O(\{ {\bf v(k)}\},U_{\bf v}(k))\]
of $Aut(P)$ contains only the identity from $G$, where ${\bf v}(i)$ is any
vertex in $P$ not fixed by $g(i)$, thus $G$ is discrete after all.
If $h$ is not an embedding of $G$ then certainly it is required that
$Ker(h)$ is finite for stabilisers in $G$ of points to be finite. Moreover 
from above $h(G)$ being indiscrete
implies that $h(G)_{\bf v}$ is infinite so certainly $G_{\bf v}$ is infinite.

However  we see from this that (v) does not imply (iii), for instance
by taking any infinite group $G$ and letting it act on its star graph $S$, 
which is one vertex for each element of $G$ and an extra root vertex $v_0$. 
We then join every vertex $v\neq v_0$ to $v_0$ by an edge (so $S$ is indeed
a tree) and let $G$ act on the
vertices $v\neq v_0$ by left multiplication, but $v_0$ is a global fixed
point. Thus the action does not satisfy any of (i) to (iv), but as all 
vertices except the root have trivial stabiliser it follows that $G$ is 
discrete in $Aut(S)$.
\end{proof}
  
\begin{co}
If in Proposition \ref{eqpr} 
each tree is locally finite then all five conditions are equivalent.
\end{co}
\begin{proof}
If $T_1,\ldots ,T_n$ are all locally
finite then they are proper metric spaces, so $P$ will be too, hence (i)
and (ii) are the same in this case because all closed bounded subsets of 
$P$ will be compact. 

Now suppose that (v) holds, but again let us first
suppose that $G$ embeds in $Aut(P)$ via $h$. Then \cite{rat} Theorem 5.3.5
shows that properness of the action (i.e. definition (ii), although called
a discontinuous action there) is equivalent to discreteness
where this is shown for the group of isometries of any proper metric space,
under the compact open topology. If $G$ does not embed in $Aut(P)$ under $h$
then $G/Ker(h)$ embeds in $Aut(P)$ as $h(G)$, 
with discreteness of $h(G)$ implying that
$h(G)_{\bf v}$ is finite for any vertex ${\bf v}\in P$ from the above.
As $h(G)$ is a quotient of $G$ with finite kernel, each
element of $G/Ker(h)$ is a coset in $G$ consisting of finitely many $g\in G$
so $G_{\bf v}$ is finite too.
\end{proof}

Hence we will adopt the following notion:
\begin{defn} We say a group $G$ acts {\bf properly} on a finite product $P$
of trees (locally finite or otherwise) if the stabiliser of every
vertex in $P$ is a finite group.
\end{defn}
This is surely the right definition in the locally finite case, where all
the definitions above are equivalent but this one is the easiest to check
in practice. We also believe that this is the right definition for general 
trees: certainly
it seems that (v) is simply too weak to be of use. Moreover definition (i) 
comes over as rather strong, although in some circumstances this might be the
appropriate concept to use, such as when looking at the Haagerup property.
 
If our group $G$ acts preserving factors so that any $g\in G$ can be written
$(g_1,\ldots ,g_n)$ for some $g_i\in Aut(T_i)$ then 
$g$ is in the stabiliser $G_{\bf v}$ for a vertex ${\bf v}=(v_1,\ldots ,v_n)$
of $P$ if and only if $g_1$ stabilises 
$v_1\in T_1,\ldots ,$ and $g_n$ stabilises 
$v_n\in T_n$. In other words, given a vertex 
$v_i\in T_i$, suppose we write $G_{v_i}$ for the stabiliser of $v_i$ in $G$
when we regard $G$ as acting on the tree $T_i$ by projection. Then $G$ acts 
properly if and only if for all vertices $(v_1,\ldots ,v_n)\in P$, we have the 
intersection of stabilisers $G_{v_1}\cap \ldots \cap G_{v_n}$ is finite.
Usually we will be interested in the case where $G$ is a torsion free group,
thus $G$ acts properly on a finite product $P$ of trees if and only if it acts
freely on the vertices of $P$ (and indeed if and only if $G$ acts freely on $P$
because if some element $g\in G$ stabilises a cube then a power of $G$ 
will stabilise its vertices). Using this, we can give a basic criterion
for such groups to act properly on $P$ if they preserve factors,
which nevertheless is very useful
because it involves looking at the components of the
elements of $G$ rather than the vertices
of $P$.
\begin{prop} \label{hyel}
Suppose that $G$ is a torsion free group acting on a finite product 
$P=T_1\times \ldots \times T_n$ of
trees preserving factors. Then $G$ acts properly on $P$
if and only if for any non identity $g=(g_1, \ldots , g_n)\in G$ with
each $g_i$ regarded as an element of $Aut(T_i)$ under the natural
projection, we have that at least
one $g_i$ from $g_1,\ldots ,g_n$ acts as a hyperbolic isometry on the
tree $T_i$.
\end{prop}
\begin{proof} If we have a non identity element $g=(g_1,\ldots ,g_n)\in G$   
where each $g_i$ acts as an elliptic element on $T_i$ then there will
be a vertex $v_i\in T_i$ fixed by $g_i$ so that the point 
$(v_1,\ldots ,v_n)\in P$ is fixed by $g$ and all its powers, thus the
action is not proper. Conversely if $g_i$ acts as a hyperbolic element
on $T_i$ then no vertex $v_i\in T_i$ is fixed by $g_i$, so no vertex
${\bf v}=(v_1,\ldots ,v_k)$ in $P$ is fixed by $g$.
\end{proof}  

If $G$ acts on a finite product of trees without preserving factors then
Proposition \ref{hyel} is still useful, because as mentioned earlier
for $H\leq_f G$ we have that $G$ acts properly if and only if $H$ acts 
properly, so we just take 
any finite index subgroup $H$ preserving factors and apply this
Proposition to $H$.

\section{Groups acting properly on a product of trees}
The material in the last section now begs two questions: which groups act 
properly on a finite
product of trees and which groups act properly on a finite
product of locally finite trees? Of course both of these properties are
preserved by taking subgroups (as is the property of acting properly 
preserving factors in both cases). However if $H$ has finite index $i\geq 2$ in 
$G$ and $H$ has either of these properties then $G$ has the respective property 
too by induced actions, namely if $H$ is acting on $P$ then we define an 
action of $G$ on the product $P^i$ which is also isometric
and proper, though $G$ will never act preserving factors, even
if $H$ does. Overall this implies (if we allow proper actions which do
not preserve factors) that both of our two properties are
commensurability invariants.  

Starting with the case of one tree, we have mentioned that $G$ has a free 
action on some tree if and
only if $G$ has a free action on some locally finite tree if and only
if $G$ is a free group. Moreover if $G$ is finitely generated then
we can replace free action with proper
action and free group with virtually free group in the above statement.
If however $G$ is not finitely generated then 
given a proper action of $G$ on a tree, we can only say that $G$ is
locally virtually free (though not all of these groups have such an
action, for instance take $\Q$ which is torsion free).

Moving to the product of more than one tree, the next group we might
examine is the fundamental group $S_g$
of the closed orientable surface $\Sigma_g$  for genus $g\geq 2$. This is
known to act properly by considering Bass - Serre trees obtained from
splittings of a system of curves that fill the surface,
though here is a basic argument utilising
Proposition \ref{hyel}.\\
\hfill\\
{\bf Example}: 
Taking $g=2$ (as $S_g$ is a subgroup of $S_2$ for $g\geq 2$), we have that
$S_2$ acts on its Bass - Serre tree $T_1$ via
the amalgamation $F_2*_{\Z} F_2$ for $\langle a,b\rangle$ the first $F_2$
factor and $\langle c,d\rangle$ the second, where we have $[a,b][c,d]=id$.
Although $T_1$ is of course
not locally finite, the elliptic elements of this action are all
conjugate into the vertex groups $\langle a,b\rangle$ or $\langle c,d\rangle$.
Now consider the homomorphism from $S_2$ onto another rank 2 free group
$\langle x,y\rangle$ given by $a,d\mapsto x$ and $b,c\mapsto y$. As
$\langle x,y\rangle$ acts freely and purely hyperbolically on a locally
finite tree $T_2$ we have that our surface group $S_2$ acts on $T_2$ where
any non identity element in $\langle a,b\rangle$ or $\langle c,d\rangle$
is hyperbolic. This is also true for any non identity element conjugate in
$S_2$ into $\langle a,b\rangle$ or $\langle c,d\rangle$ so by Proposition
\ref{hyel} we have that $S_2$ (and $S_g$ for $g\geq 2$) acts properly 
preserving factors on a
product of two trees $T_1\times T_2$,
where the second is locally finite but the first is not.

The question of whether such an action exists when all factor trees are
locally finite was raised in \cite{flss} and seems an important and
difficult question. We will look at this in the final section, but for
this section our aim is to reaffirm that many groups act properly
on a finite product of trees,  whereas we will give evidence in the
next section which suggests that acting properly
on a finite product of locally finite trees is much rarer.

Let us first think about what obstructions we know of that prevent a group
$G$ from acting properly on a finite product $P$ of trees. As $P$ will
be a CAT(0) space and the action of $G$ will be proper and semisimple,
we already have group theoretic restrictions on $G$ 
(see \cite{bh} Part III, Chapter $\Gamma$, Theorem 1.1), thus
ruling out such examples as mapping class groups of genus $g\geq 3$.
Further $P$ is not just a CAT(0) space but a CAT(0) cube complex 
(indeed finite dimensional though it is only locally finite when the
trees are locally finite), which can provide further restrictions.
An infinite group $G$ can have Serre's fixed point property (FA) and still
have a proper action on a finite product of trees but clearly there cannot
be such an action preserving products, thus this rules out infinite groups where
every finite index subgroup has property (FA). This is known as hereditary
Property (FA), for instance see \cite{deck}.
In particular $G$
cannot have property (T), nor contain an infinite subgroup with property (T).

Having mentioned some negative results, which groups might we hope do act
properly on a finite product of trees (where for the rest of this
section we do not assume
local finiteness)? Surely our main case of interest should be RAAGs, as if our 
property holds here then it will immediately hold for all groups which 
virtually embed in a RAAG, thus covering all virtually special groups.
This is true and follows from \cite{drnjan} where the same property
is shown for Coxeter groups, by utilising their action on the Davis complex.
Here we can give a quick and basic proof directly for RAAGs, by again
using Proposition \ref{hyel}.
\begin{thm} \label{rg}
If $\Gamma$ is a finite graph with $n$ vertices then the
RAAG group $G(\Gamma)$ defined by this graph acts properly and preserving
factors on a product of $n$ trees.
\end{thm}
\begin{proof}
The proof is by induction on $n$. Given a finite graph $\Gamma$ built
by adding one vertex at a time and all the relevant edges, suppose that
we go from $\Gamma_n$ with $n$ vertices to $\Gamma_{n+1}$ with $n+1$ vertices
by adding the vertex $v$.
Then it is well known that in terms of groups, we have $G(\Gamma_{n+1})$
is an HNN extension $\langle G(\Gamma_n),t\rangle$ with stable letter $t$,
so that $tAt^{-1}=B$ for some associated subgroups $A,B$ of $G(\Gamma_n)$. 
However the
crucial point here is that $A=B$ (the subgroup of $G(\Gamma_n)$ obtained
from the induced graph in $\Gamma_n$ of the neighbours of $v$) and that
conjugation by $t$ is acting as the identity.

Thus from this HNN extension, we have that $G(\Gamma_{n+1})$ acts on the
Bass - Serre tree $T$ and any elliptic element in this action is
conjugate into the subgroup $G(\Gamma_n)$. Now by induction
this subgroup $G(\Gamma_n)$ of $G(\Gamma_{n+1})$ acts properly on the product
of $n$ trees $T_1\times \ldots \times T_n$, preserving factors.
However the above
description of the HNN extension means that $G(\Gamma_n)$ is a retract of
$G(\Gamma_{n+1})$, so $G(\Gamma_{n+1})$ quotienting onto $G(\Gamma_n)$ means that
$G(\Gamma_{n+1})$ also acts on $T_1\times \ldots \times T_n$ preserving
factors. Moreover, as this action restricted to $G(\Gamma_n)$ is proper,
any non identity element of $G(\Gamma_n)$ acts hyperbolically on at least
one of the trees $T_1,\ldots ,T_n$. Now $G(\Gamma_{n+1})$ also acts on
the product $T_1\times \ldots \times T_n\times T$ preserving factors.
Then for any non identity element $g\in G(\Gamma_{n+1})$, either $g$ is not
conjugate into $G(\Gamma_n)$ and so acts hyperbolically on $T$, or a conjugate
of $g$ acts hyperbolically on one of $T_1,\ldots ,T_n$ thus
so does $g$. Hence we have a proper action, preserving factors,
by Proposition \ref{hyel}.
\end{proof}

It is clear from the proof of Theorem \ref{rg} that requiring $n$ trees
for a graph of $n$ vertices is wasteful in certain cases, so let  
us now show that a similar argument gives us a much better bound in general.
This will be useful in giving an exact description of those RAAGs which act
properly on a product of two trees.

\begin{co} \label{chr}
Let $\Gamma$ be a finite graph with chromatic number $k$ then the associated
RAAG $G(\Gamma)$ acts properly preserving factors on the product of $k$ trees.
\end{co}
\begin{proof}
Let the vertices of $\Gamma$ be given the colours $1,\ldots ,k$ and for $i$
between 1 and $k$, let
the graph $\Delta_i$ be the induced subgraph on all vertices which are
coloured any of $1,\ldots ,i$. Then similarly to the above we have that
for $2\leq i\leq k-1$ the group $G(\Delta_{i+1})$ admits a graph of groups
decomposition with one vertex, corresponding to the vertex group $G(\Delta_i)$
and $s$ self loops, where $s$ is the number of vertices being added to 
$\Delta_i$ to obtain $\Delta_{i+1}$. When we add one of these vertices $v$,
all edges from $v$ that also get added have their other endpoint back in
the graph $\Delta_i$, thus $v$ introduces one stable letter corresponding
to one of these self loops, and again the edge groups at either end are
the same subgroup of $G(\Delta_i)$ with the stable letter acting trivially
by conjugation. Thus we still have that $G(\Delta_i)$ is a retract of
$G(\Delta_{i+1})$.

Thus on assuming that $G(\Delta_i)$ acts properly on the product of $i$ trees
(with the base case being that the graph $\Delta_1$ has no edges so 
$G(\Delta_1)$ is free), we introduce the Bass - Serre tree of this
graph of groups splitting of $G(\Delta_{i+1})$ in which all elliptic elements
are conjugate into $G(\Delta_i)$. Moreover $G(\Delta_{i+1})$ can be made to
act on the first $i$ trees using the retract property, so that $G(\Delta_i)$
still acts properly on this product. Thus on adding the Bass - Serre tree
to our $i$ trees already present, we have that $G(\Delta_{i+1})$ acts 
preserving factors on the
product of $i+1$ trees in which every non identity element acts hyperbolically
on one of the component trees.
\end{proof}

We now obtain
useful obstructions for a group $G$ to act properly on a 
product of a certain number of trees
if $G$ contains non cyclic free abelian subgroups $\Z^n$ for $n\geq 2$
(at least if $G$ acts preserving factors, but if not then we will have 
some $H\leq_f G$ which does act this way and which also contains $\Z^n$). 
This works especially
well when $G$ has infinite order elements $x,y,z$ where $x$ commutes
with both $y$ and $z$ but $yz\neq zy$. Examining elements with ``large''
centralisers gives us a pair of lemmas which are presumably folklore
(for instance see \cite{cuvo} Section 1 for similar results)
but which come in very handy for these groups.

\begin{lem} \label{cent}
For an infinite order element $x$ in a group $G$ acting on a tree
$T$ with $x$ a hyperbolic element, the centraliser $C_G(x)$ fixes the
axis $A_x$ of $x$ setwise and preserves its direction. Consequently we
obtain a non trivial homomorphism $\theta:C_G(x)\rightarrow\Z$ where for
$g\in C_G(x)$ we have that $g$ is elliptic if and only if $g\in Ker(\theta)$.
\end{lem}
\begin{proof}
For $g\in G$ the axis $A_{gxg^{-1}}$ of $gxg^{-1}$ is of course $g(A_x)$, so
$g(A_x)=A_x$ if $gxg^{-1}=x$, whereas if $G$ reversed the direction of $A_x$
then we would have $gxg^{-1}=x^{-1}$. The homomorphism $\theta$ is just the
translation length of $g\in C_G(x)$ on $A_x$, which determines whether
$g$ acts hyperbolically or elliptically on $A_x$. However this is the
same as the type of action of $g$ on the whole tree $T$.
\end{proof}
Thus for free abelian groups of finite rank, and especially $\Z\times\Z$,
we have:
\begin{lem} \label{z}
Suppose that $G=\Z^n$ for $n\geq 2$ acts on a tree without a global
fixed point. Then there exists a non trivial homomorphism 
$\theta:G \rightarrow \Z$ such that $g\in G$ is elliptic if and only if
$g$ is in $Ker(\theta)$. In particular if $G=\Z\times\Z$ and we have
$g,h\in G$ which are both elliptic but which do not both lie in the same maximal
cyclic subgroup then $G$ acts with a global fixed point.
\end{lem}
\begin{proof} There must be some hyperbolic element in $G$ by Serre's
Lemma for finitely generated groups of elliptics, whereupon the first
part follows by Lemma \ref{cent}. The second part holds because
$\langle g,h\rangle$ will generate a finite index subgroup of $G$ on which
$\theta$ is zero.
\end{proof}      

Note that this implies that $\Z^n$ cannot act properly (with or without
preserving factors) on a product of $m<n$ trees, although of course
it does act properly on $\R^n$.

So let us now ask the question: which RAAGs act properly 
on the product of two trees? A necessary condition here is that the
RAAG is a 2 dimensional group: indeed a RAAG
has dimension (say cohomological, or geometric) greater than 2 if and
only if the defining graph contains a triangle. This is because if so
then the RAAG contains $\Z^3$ whereupon it cannot act properly on the
product of two trees, whether locally finite or not and whether preserving
factors or not. On the other hand
the universal cover of its Salvetti complex, which is a CAT(0) cube complex
on which the RAAG acts properly and even cocompactly, will be 2 dimensional
if the defining graph is triangle free. With this in mind, optimists
might expect that all 2 dimensional RAAGs act properly on a product
of two trees whereas pessimists might assume that only free groups and
the direct product of two free groups do, on the grounds that a product
of two trees is much more restrictive than a general 2-dimensional
CAT(0) cube complex.
In fact we will now see that the answer can be said to lie halfway between
these two extremes.

\begin{prop} \label{odd} 
If $G$ is a RAAG whose defining graph is a cycle of odd length then $G$ does
not act properly (with or without swapping factors) on any product of two
trees (whether locally finite or not).
\end{prop}
\begin{proof} We first assume that $G$ acts properly by preserving
the two factors $T_1,T_2$. Let the standard generators of $G$ be
$x_1,\ldots ,x_c$ for $c$ odd with all subscripts taken modulo $c$, so
that $x_i$ commutes with $x_{i-1}$ and $x_{i+1}$, generating a copy
of $\Z^2$ in each case. As discussed above, we
can assume that $c>3$. Now consider the group $\langle x,y\rangle=\Z^2$
acting properly on only two trees. From Lemma \ref{z}, if both $x$ and $y$
act as elliptic elements on one tree then so does all of $\langle x,y\rangle$,
so we would require every element of $\langle x,y\rangle$ to act hyperbolically
on the other tree but this is impossible, even if $x$ and $y$ are themselves
both hyperbolic elements.

This implies for $G$ that in each of the two tree actions, we do not have
two successive generators $x_i,x_{i+1}$ which are both elliptic elements
in the same action. But of course for a proper action we must have that
each $x_i$ acts hyperbolically in at least one of the two actions.
Consequently there must exist some $j$ where a generator
$x_j$ is hyperbolic in both actions, as otherwise we would have to alternate
between hyperbolic and elliptic elements in one action and the reverse
in the other, but this cannot occur because $c$ is odd.

As the neighbouring generators $x_{j-1},x_{j+1}$ lie in the centraliser
$C_G(x_j)$, Lemma \ref{z} tells us that for each of the two actions
we have homomorphisms $\theta_1,\theta_2:C_G(x_j)\rightarrow \Z$ with
the kernels consisting entirely of elliptic elements. Now $x_{j-1}$ and
$x_{j+1}$ need not themselves be elliptic (for instance the homomorphism
could send each of $x_{j-1},x_j,x_{j+1}$ to 1) but $[x_{j-1},x_{j+1}]$ certainly
will be elliptic in either action and this is an infinite order element
of $G$ if $c>3$
(for instance, use the homomorphism from $G$ to $F_2$ sending
$x_{j-1},x_{j+1}$ to a generating pair and other generators to the identity).

To rule out $G$ acting permuting factors, we will use a very useful trick
for RAAGs, as observed in \cite{kk}, which is that for any $H\leq_f G$,
there exists $k\geq 1$ such that for all of our standard generators $x_i$
we have $x_i^k\in H$. Therefore if $G$ acted properly, we could take such an
$H$ preserving factors. But now the argument runs exactly as before when
applied to $\langle x_1^k,\ldots ,x_c^k\rangle$ in $H$, rather than
$x_1,\ldots ,x_c$ in $G$.
\end{proof}

\begin{co} \label{oddc}
If $G$ is a RAAG with defining graph containing an odd length
closed path then $G$ does not act properly (whether preserving factors or
otherwise) on the product of two trees (whether locally finite or not).
\end{co}
\begin{proof} We just need to show that $G$ contains a subgroup as in
Proposition \ref{odd}. First of all we can assume the closed path is an
embedded cycle, because on splitting the path in two at an intersection
point, one of these closed subpaths must have odd length. Next we can
assume that this cycle is actually induced, else we can add in another edge
to the cycle from the defining graph and split to obtain two shorter cycles,
one of which has odd length. On continuing, we end with an induced cycle
of odd length at least 3 so the corresponding RAAG is a subgroup of $G$.
\end{proof}

A much studied question is when one RAAG is a subgroup of another, see
\cite{kk} and other papers cited there for a range of results. Using the
above, we can obtain a quick and straightforward result for ourselves:
\begin{co} \label{sbgr}
A RAAG with defining graph having a closed path of odd length cannot
be a subgroup of a RAAG with defining graph having no closed path of odd length.
\end{co}
\begin{proof} The second graph has chromatic number two and so acts properly
on a product of two trees preserving factors, by Corollary \ref{chr}. This
property is clearly preserved by subgroups but the first group does not have
such an action by Corollary \ref{oddc}.
\end{proof}

Intriguingly, we have not seen this result stated directly in the literature
though it can be deduced from known work, such as Theorem 1.11 in \cite{kk}
which states that for two finite graphs $\Gamma, \Delta$ with $\Delta$ triangle
free, the RAAG $G(\Gamma)$ embeds in $G(\Delta)$ if and only if $\Gamma$
is an induced subgraph of the extension graph of $\Delta$. However this 
extension graph has the same chromatic number as $\Delta$, so we recover
the above by taking $\Gamma$ to be an odd length cycle (of length at least 5) 
and $\Delta$ to have no closed path of odd length, thus nor does its extension
graph. (We thank the second author for indicating this argument.)
\section{Proper actions in the locally finite case}

We now consider the case where our trees are all locally finite. 
We first point out that if we are dealing with finitely generated
groups then we do not need to worry about the distinction between
our various types of locally finite trees.

\begin{lem} \label{equiv}
If $G$ is a finitely generated group then $G$ acts
properly on a finite product of locally finite trees if and only
if $G$ acts properly on a finite product of uniformly bounded trees
if and only if $G$ acts properly on a finite product of $d$-regular
trees for some finite $d$.
\end{lem}
\begin{proof}
First suppose that the trees $T_1,\ldots ,T_k$
are locally finite.
We can suppose that $G$ preserves factors by dropping down
to a finite index subgroup and then taking an induced action at the
end (which will preserve properness of the action).
Let us therefore consider the action of $G$ on each 
$T_i$. As $G$ is finitely
generated, its action on the core $C(T_i)$ has quotient 
$G\backslash C(T_i)$ which is a finite graph and this clearly has bounded
valence. Moreover the vertices in $C(T_i)$ fall into
finitely many orbits and within each orbit any vertex can be moved to any
other vertex by an automorphism of $C(T_i)$, so they have the same
finite degree (even though this could be much bigger than the degree of the 
corresponding vertex in the quotient graph), thus we have a
finite upper bound for the degree of vertices in $C(T_i)$.

We then replace each $T_i$ in the product with its core $C(T_i)$,
which is invariant under $G$ and with the action on
$C(T_1)\times\ldots\times C(T_k)$ still proper.

Now suppose that the trees are uniformly bounded. We can use the construction
as described at the start of Section 3 to turn each tree $T$
into the $d$-regular
tree $T_d$ for some $d$ by adding rooted trees at the vertices with degree
less than $d$, and $G$ will also act on $T_d$. Moreover if $G$ acts
properly on the the original product $T_1\times \ldots \times T_k$ then
it will also do so on the new product $T_d\times\ldots\times T_d$ of
regular trees. This is because any new vertex $v$ added to $T$ lies in
a rooted tree whose root $w$ was already a vertex of $T$, so that if
an automorphism of $T$ fixes $v$ then it also fixes $w$.

The final step is clear because all $d$-regular trees are locally finite.
\end{proof}

Let us consider which groups $G$ we already know act properly on a finite 
product ($k$ say)
of such trees. On assuming that $G$ acts preserving factors, we have  
direct products of $k$ free groups, where the $i$th factor group acts purely 
hyperbolically on the $i$th factor tree and trivially on all other factor 
trees. Other examples for $k=2$ include the Burger - Mozes
groups and lamplighter groups of the form $C_p\wr \Z$, as well as all 
subgroups of the above. If we now wish to find a group which does not
act properly and preserving factors in the locally finite case but which
does do so in the general case, we will need some means of distinguishing
between these two situations. The following straightforward Proposition is 
the only such result we will utilise here.

\begin{prop} \label{aclf}
Suppose that a group $G$ is acting on a locally finite tree $T$ where we have
two elements $g,h$ which act elliptically. Then there is
$n\in\N$ such that $g^n$ and $h$ have a common fixed point,
and so $\langle g^n,h\rangle$ consists entirely of elliptic elements.
\end{prop} 
\begin{proof}
Suppose that $g$ fixes the vertex $v\in T$ and $h$ fixes the vertex $w$.
Then a power of $g$ will fix the edges incident at $v$, then a
further power of that will fix the vertices distance 2 away from $v$ and
so on. Consequently for some large $n$ we get that $g^n$ also fixes $w$
and so $\langle g^n,h\rangle$ has the global fixed point $w$.
\end{proof}

An interesting class of groups for testing out various conjectures
are the free by cyclic groups $F_n\rtimes_\alpha\Z$, some of which are
CAT(0) and some of which are not. Here we consider the group 
$G=F_2\rtimes_\alpha\Z$ where $F_2=\langle a,b\rangle$ and 
$\alpha(a)=ab$, $\alpha(b)=b$, so that
\[G=\langle a,b,t|tat^{-1}=ab,tbt^{-1}=b\rangle.\]
This group is not word hyperbolic but does
act properly and cocompactly on a 2 dimensional cube complex (for instance
consider the alternative presentation $\langle s,x,y|[x,y],sxs^{-1}=y\rangle$
where $s=a^{-1}, x=t, y=bt=tb$ and draw out the two squares, noting that all
8 vertices are identified, and use Gromov's link condition). 
We first examine the action of $G$ on a product
of trees in the general case. 

\begin{prop} \label{thrtr}
The group
\[G=\langle a,b,t|tat^{-1}=ab,tbt^{-1}=b\rangle\]
acts properly preserving factors on a product of three trees but not on a
product of two trees.
\end{prop}
\begin{proof}
We have $a^{-1}ta=tb$ with $\langle t,tb\rangle\cong \Z^2$. Thus in any action
of $G$ on a tree, either $t$ and $tb$ are both elliptic or both hyperbolic.
This means that in any proper action preserving factors of $G$ on the
product of two trees, we must have that $t$ is hyperbolic in each of the
two actions or else some element of $\langle t,tb\rangle$ will be elliptic 
in both. But if this is so then, as $aba^{-1}$ and $b$ both commute with $t$,
the element $[aba^{-1},b]$ will be elliptic in each action.

To obtain an action with three trees, we consider the other presentation
$\langle s,x,y|[x,y],sxs^{-1}=y\rangle$ which expresses $G$ as an HNN
extension in which every elliptic element is conjugate into 
$\langle x,y\rangle\cong\Z^2$. But the homomorphism $\theta:G\rightarrow\Z$ 
given by $\theta(s)=0,\theta(x)=1,\theta(y)=1$ gives rise to an action of $G$
on $\R$ where the only elliptic elements of $\langle x,y\rangle$ are those
of the form $x^iy^{-i}$.

But we also have a homomorphism $\phi$ from $G$ to the infinite dihedral 
group $D_\infty$ acting on $\R$, where $\phi(s)$ is the map $p\mapsto -p$, and
$x,y$ are sent to the translations $p\mapsto p+1$, $p\mapsto p-1$ respectively.
Then if $i\neq 0$ this will send $x^iy^{-i}$ to the hyperbolic element
$p\mapsto p+2i$.
\end{proof}

We can now show a strong negative result for this
group in the locally finite case. 

\begin{thm} \label{noac} 
The group \[G=\langle a,b,t|tat^{-1}=ab,tbt^{-1}=b\rangle\]
cannot act properly on any finite product $T_1\times\ldots\times T_k$
of locally finite trees preserving factors.
\end{thm}
\begin{proof} The group $G$ is torsion free, so suppose that the element $t$
acts hyperbolically in the action of $G$ on some factor tree $T_i$. Now
$aba^{-1}$ and $b$ commute with $t$ in $G$ but 
$\langle aba^{-1},b\rangle\cong F_2$. So by Lemma \ref{cent} we have that
for all $m,n>0$ the infinite order element $[ab^ma^{-1},b^n]$ acts elliptically
on the tree $T_i$.

Now suppose $t$ acts elliptically in the action of $G$ on some other factor
tree $T_j$. As $a^{-1}ta=bt$, we have that $bt$ acts elliptically too and thus
so does $b$ by Lemma \ref{z}. Again $aba^{-1}$ is also elliptic, so by the
local finiteness of $T_j$ and Proposition \ref{aclf} there is some
$n_j>0$ such that $ab^{n_j}a^{-1}$ and $b^{n_j}$ share a fixed point. This
also holds for all multiples of $n_j$, so for all $k>0$ we have that
$[ab^{kn_j}a^{-1},b^{kn_j}]$ is elliptic when acting on $T_j$. Thus if we take
$N$ to be the product of the $n_j$s over all $j$ where $t$ acts elliptically
on $T_j$, we find that the element $[ab^Na^{-1}, b^N]$ is elliptic in all of
the product actions where $t$ acts elliptically, and indeed when $t$ acts
hyperbolically too.
\end{proof} 

We end this section by discussing the question of whether this group $G$
can act properly on a finite product of locally finite trees if it is
allowed to permute the factors. On taking a finite index subgroup $H$
of $G$ which we suppose does act properly on a finite product of locally 
finite trees preserving factors, let us first suppose that the element
$a\in H$. Now $t$ and $b$ need not be in $H$ but there will exist some
$i>0$ such that $t^i$ and $b^i$ are. These elements will also commute
with each other and $t^i$, $(tb)^i$ are still conjugate in $H$ under
our assumption. Consequently the proof of Theorem \ref{noac} goes through
just as before on replacing $G$ with $H$ and $t,b$ with $t^i,b^i$.

The problem is that once $a\notin H$, we cannot assume $t^i$ and
$(tb)^i$ are conjugate in $H$. In particular consider the index 2
subgroup $N$ of $G$ given by the kernel of the homomorphism to $C_2$
sending $a$ to 1 and $b,t$ to 0. A presentation for $N$ is
\[\langle c,x,y,z|[x,y],[y,z],cxc^{-1}=z\rangle \]
where $x,y$ are the same elements as in the alternative presentation for
$G$ above and $c=s^2$. We therefore ask:\\
\hfill\\
{\bf Question}: Does the group $N$ act properly on a finite product of
locally finite trees preserving factors?\\ 
A no answer would have the
following strong consequences for RAAGs: 

\begin{co} \label{noacc}
If $N$ has no such action then any RAAG whose defining graph
has $\,\bullet
{\bf -}\bullet{\bf -}\bullet{\bf -}\bullet$ as an induced subgraph has
no proper action on a finite product of locally finite trees, with or
without permuting factors.
\end{co}
\begin{proof} The group $N$ embeds in the RAAG with the above defining
graph by \cite{nw} Theorem 1.2, which 
in turn embeds in any RAAG $R$ having such an induced 
subgraph. Thus if this question has a negative answer then $R$ does not
act properly on a finite product of locally finite trees preserving factors.
But now we use the fact mentioned earlier that every finite index subgroup
of $R$ contains an isomorphic copy of $R$ to rule out actions that permute
factors.
\end{proof}

Given the results mentioned in the previous section, this would demonstrate
a big difference for most RAAGs in how they can act on a general product
of trees as opposed to how they do in the locally finite case.

\section{Surface groups and other hyperbolic 
groups}  
In this section we consider the intriguing question, raised in \cite{flss},
of whether a closed hyperbolic surface group $S_g$ (or indeed any non
elementary word hyperbolic group $H$ other than a virtually free group)
acts properly on a finite product of locally finite trees. In this section
we consider the case of groups $G$ which are torsion free but
which do not contain $\Z\times\Z$. It is not hard to show (for instance
by induction on $k$) that if
$G$ acts faithfully on the
product of $k$
trees $P=T_1\times\ldots\times T_k$ preserving factors
then there exists $i$ such that
the projection of $G$ to $Aut(T_i)$ is injective. Now suppose that the same
group $G$ acts properly (hence faithfully, as $G$ is torsion free) on the
product of $k$
trees $P=T_1\times\ldots\times T_k$ preserving factors. An easy adaptation
of the above result is that there exists $i$ such that
$G$ not only acts faithfully on $Aut(T_i)$ but also without a global
fixed point (as one can first remove the factors where $G$ acts with a
global fixed point and this will still preserve
properness of the action on the product).

Thus failure to have a faithful action on a single locally finite tree without 
a global fixed point is therefore an obstruction to $G$ acting properly on a 
finite product of locally finite trees. If $G$ is finitely generated then
the second part of Theorem \ref{fthtre} characterises such groups.
However this does not rule out hyperbolic surface groups: 
as $S_g$ is residually finite, we 
merely need to find some action (not necessarily faithful) of $S_g$ on a
locally finite tree without a global fixed point. We can stick to
$S_2=\langle a,b,c,d|[a,b][c,d]\rangle$ as $S_g$ is a finite index subgroup
for $g\geq 2$ and this can be achieved by taking a surjective homomorphism
from $S_2$ to a group known to have this property, such as $F_2$ or a free
product $C_m*C_n$.\\

However in this section we have not yet used the fact that here our trees
are all locally finite, so we do so now.

\begin{thm} \label{injcloc}
Suppose the torsion free group $G$ acts properly on the
product of $k$ locally finite
trees $P=T_1\times\ldots\times T_k$ preserving factors
but does not contain $\Z\times\Z$. Then we can restrict the action
of $G$ to a subproduct $Q=T_{i_1}\times \ldots \times T_{i_j}$
(for $i_1<\ldots <i_j$) of these trees such that the action of $G$ on $Q$
is still proper and such that every projection of $G$ to $Aut(T_{i_1}),
\ldots , Aut(T_{i_j})$ is injective and does not have a global fixed
point.
\end{thm}
\begin{proof}
We first remove any tree where the
projection action of $G$ has a global fixed point, but having done that
we also remove in turn any tree $T_i$ where the action of $G$ on $T_i$ has the
following property: for any element $g\in G$ which is hyperbolic in this
action, there is some other tree left in the product (maybe depending on $g$) 
where $g$ also acts hyperbolically. On removing this tree $T_i$ from the
product, we have by Proposition \ref{hyel} that $G$ is still acting
properly and we continue until no such tree is left, whereupon we revert
to the original notation $T_1\times\ldots\times  T_k$ for the final product
of trees thus obtained.

Now suppose that (renumbering these trees if necessary) the action of
$G$ on $T_1$ is not faithful, so there exists an infinite order element
$w\in G$ acting trivially on $T_1$. But there will exist other elements of
$G$ acting hyperbolically on $T_1$ or else it would have been removed in the
first stage above. Furthermore there is some element $g\in G$ acting
hyperbolically on $T_1$ and such that the action of $G$ is elliptic on all
of $T_2,\ldots ,T_k$ (or else $T_1$ would have been removed at some point
during the second stage). This means that no positive power $g^n$ can
commute with $w$ in $G$, as otherwise $\langle g^n,w\rangle\cong\Z\times\Z$
unless $g^{rn}=w^s$ for some non zero $r,s$ but this would imply that the
action of $g$ on $T_1$ has finite order.

Hence on $T_2,\ldots ,T_k$ we have that $g$ is elliptic, as is $wgw^{-1}$.
So far we have not used local finiteness of our trees but now we observe that
for each $2\leq j\leq k$ there is $n_j$ such that $g^{n_j}$ and
$wg^{n_j}w^{-1}$ have a common fixed point, as in Proposition \ref{aclf}.
Consequently
$wg^{n_j}w^{-1}g^{-n_j}$ is also elliptic in the action on $T_j$. But this
argument also applies to any multiple of $n_j$, thus for $N=n_2\ldots n_k$
we have that $wg^Nw^{-1}g^{-N}$ is elliptic on $T_2,\ldots ,T_k$. Moreover
on $T_1$ this element acts as the identity because $w$ does, but we noted above
that it is not the identity element in $G$. So the action of $G$ on
$T_1\times\ldots\times T_k$ is not proper by Proposition \ref{hyel}.
\end{proof}

Note: this result fails even for two trees if one of them is not locally
finite, as shown in the example at the start of Section 5.

At first glance this does not seem to improve on our obstruction above.
However if such an action of $G$ exists as in the statement of
Theorem \ref{injcloc} then for any non identity $g\in G$, we can
pick a tree $T_i$ from the product on which $g$ acts hyperbolically.
Thus we can say: if a group $G$ is torsion free, does not contain $\Z\times\Z$
and acts properly on a finite product of locally
finite trees then for every non identity $g\in G$, there exists an
faithful action of $G$ on a locally finite tree $T_g$ on which $g$ acts
hyperbolically. Again the surface group $S_2$ is easily seen to avoid
this obstruction. 
For instance as it is a residually free group we can pick any non identity
$g\in S_2$ and take a homomorphism from $S_2$ to a free group $F_r$
which does not vanish on $g$. We can then let $F_r$
act on its Cayley graph, thus $S_2$ acts on this graph too with $g$ a
hyperbolic element. Whilst this is certainly not a faithful action of
the residually finite group $S_2$,
we can now apply the proof of Theorem \ref{fthtre} to convert it
into a faithful action of $S_2$ on a bigger tree, but one that is still
locally finite.

In order to get an obstruction which is non trivial as regards the group
$S_2$, we consider the core of a tree $T$
acted on by a group $G$. If $G$ contains a hyperbolic element then the
core  $C(T)$ of $T$ is the union of all axes of hyperbolic elements in $G$
and is the unique minimal subtree invariant under the action of $G$,
whereupon we say that $G$ acts minimally on $C(T)$. Certainly the
techniques in the proof of Theorem \ref{fthtre} where rooted subtrees
are added at vertices will not result in a minimal action.
Therefore we strengthen Theorem \ref{injcloc} so that it keeps the
same hypothesis and the same conclusion except we can add that $G$
acts minimally on each tree as well: 

\begin{co} \label{injcloc2}
Suppose the torsion free group $G$ acts properly on the
product of $k$ locally finite
trees $P=T_1\times\ldots\times T_k$ preserving factors
but does not contain $\Z\times\Z$. Then for each non identity element
$g\in G$ there is an action of $G$ on some locally finite tree $T$
(depending on $g$) which is minimal, faithful and such that $g$ acts
hyperbolically.
\end{co}
\begin{proof}
We start by proceeding as in the first paragraph of the proof of Theorem
\ref{injcloc} where various trees are removed but the resulting action
of $G$ on the product is still proper. At this point though,
we now replace each of
the remaining trees $T_{i_j}$ by its core $C(T_{i_j})$, which is also
locally finite and with $G$ acting minimally on it. The action
of $G$ on this new product is proper by Proposition \ref{hyel}, because
the process of restricting an action to the core does not affect whether
an element is hyperbolic or elliptic. Now we can run through the rest
of the proof of Theorem \ref{injcloc} to get that the projection action
of $G$ on each of the remaining trees is faithful, and so given any
non identity $g\in G$, we can take the action of $G$ on one of the
factor trees where $g$ acts hyperbolically.
\end{proof}

In the next section we will show that this is true for hyperbolic surface
groups. However
we finish this section by noting that things are very different for groups 
containing 
$\Z\times\Z$. Taking $G=F_2\times\Z$ with $\Z=\langle z\rangle$, we have that
$G$ acts properly (and even cocompactly) on a product of two locally
finite trees and is also residually free. But there is no faithful
minimal action of $G$ on any tree whatsoever. This is because if $G$ acts
on a tree with $z$ a hyperbolic element then any $g\in G$ sends the axis
of $z$ to itself, so the core must be the axis of $z$ and then the 
minimal action
is not faithful. If however $z$ acts elliptically then for all hyperbolic
elements $g$, we have that $z$ sends the axis of $g$ to itself and
preserves the direction (else it conjugates $g$ to its inverse), so
$z$ must fix this axis pointwise and hence fix the whole core pointwise.

\section[Fields of positive characteristic]{Fields of positive 
characteristic and the Bruhat - Tits tree}

Question 1 of \cite{flss} asks: let $S_g$ be a closed surface group of
genus $g\geq 2$. Is there a discrete and faithful representation of
$S_g$ into $Aut(Y)$ for $Y$ a finite product of bounded valence
trees? 
As $S_g$ is torsion free and all trees considered here are
locally finite, Proposition \ref{eqpr} tells us that a discrete and
faithful representation of $S_g$ is the same as a proper action, 
whereas Lemma \ref{equiv} says it does not matter whether we use
locally finite or bounded valance trees.
Thus the question is equivalent to asking whether $S_g$ acts properly on a 
finite product of locally finite trees, and hence equivalent to whether
$S_2$ does, by induced actions.

It is pointed out in Theorem 3 of \cite{flss} that if we can find a
faithful representation of a finitely generated group $G$ into
$PGL(2,K)$ for $K$ a global field of characteristic $p>0$, say
$\F_p(x)$, then $G$ acts properly on a finite product of locally
finite trees. This is because for each valuation $v$ of $K$, the group
$PGL(2,K)$ will act faithfully on its Bruhat - Tits tree, for instance
the regular tree $T_{p+1}$ when $K=\F_p(x)$. Although this will not
be a proper action in general, the finite generation of $G$ means we
can take finitely many valuations on $K$ to get a proper action on the product
of these trees. (If $K$ is a global field of characteristic zero, namely
a number field, the above still works except that those elements of $G$
whose trace is an algebraic integer will be elliptic in every action.)

In the case where $G=S_g$, the existence of such a representation
seems a hard question, but on extending the field $\F_p(x)$ by one
transcendental element we can ask whether we have an embedding of $S_g$
using this new field. Indeed Theorems 4 and 5 in \cite{flss} 
show that for every prime $p$ at least 5, there is a faithful embedding
of $S_2$ in $PGL(2,K)$ where $K$ is a finite extension of $\F_p(x,y)$
and for any characteristic $p$ field $k$ of transcendence degree at least
2, there is a faithful embedding of $S_2$ in $PGL(n,k)$ for some $n$.
Here we will improve on these results by giving a completely explicit
faithful embedding of $S_2$ in $SL(2,K)$, and hence in $PSL(2,K)$
for $K=\F_p(x,y)$ where $p$ is any prime. To obtain faithfulness
of the representation, we use the following result of Shalen.

\begin{prop} (\cite{sha} Proposition 1.3) \label{shl}
Suppose $G_1*_HG_2$ is a free product with abelian amalgamated
subgroup $H$ and suppose we have faithful representations
$\rho_i:G_i\injects SL(2,\F)$ and
$i=1,2$ over any field $\F$ such that\\
(a) $\rho_1$ and $\rho_2$ agree on $H$,\\
(b) $\rho_1(h)=\rho_2(h)$ is diagonal for all $h\in H$ and\\
(c) For all $g\in G_1\setminus H$ we have that the bottom left hand entry
of $\rho_1(g)$ is non zero, and similarly for the top right hand entry
of $\rho_2(g)$ for all $g\in G_2\setminus H$.

Then $G_1*_HG_2$ embeds in $SL(2,\F(y))$ where $y$ is a transcendental
element over $\F$. 
\end{prop}
\begin{proof} In \cite{sha} the result was stated just for the field
$\C$ but for arbitrary dimension $d$. However the proof does work for
arbitrary fields $\F$ and general dimensions. Here we just give a summary
in the $d=2$ case, including one point in the proof which will be needed
later.

Define the representation $\rho:G_1*_HG_2\rightarrow SL(2,\F(y))$ as
equal to $\rho_2$ on $G_2$ but on $G_1$ we replace $\rho_1$ by the
conjugate representation $T\rho_1T^{-1}$ where $T$ is the diagonal
matrix $\mbox{diag}(1,y)$, and then extend to all of
$G_1*_HG_2$. Now it can be shown straightforwardly 
that any element not conjugate into $G_1\cup G_2$ is
conjugate in $G_1*_HG_2$ to something with normal form
\[g=\gamma_1\delta_1\ldots \gamma_l\delta_l\]
where all $\gamma_i\in G_1\setminus H$ and all $\delta_i\in G_2\setminus H$.
Induction on $l$ then yields that the entries of $g$ are Laurent polynomials
in $y^{\pm 1}$ with coefficients in $\F$ and with the bottom right 
hand entry of $g$ equal to $\alpha y^l+\ldots$ 
where all other terms are of
strictly lower degree in $y$. But it can be checked
that $\alpha$ is actually just a product of these respective
bottom left and top right entries, thus is a non zero element of $\F$
so this bottom right hand entry does not equal 1 and
$g$ is not the identity matrix.

Moreover the top left hand entry of $g$ is equal to a Laurent polynomial of
the form $\beta y^{l-1}$ plus lower order terms, for $\beta\in\F$ (although
unlike $\alpha$ above, $\beta$ could be zero). This means that the trace of
$g$ is also of the form $\alpha y^l$ plus lower order terms. 
\end{proof}   
  
We can use this result as follows:

\begin{co} \label{shas}
Let $\F$ be the field $\F_p(x)$. Suppose we have a pair of
2 by 2 matrices $A,B\in SL(2,\F)$ such that $\langle A,B\rangle$ is a
free group of rank 2 and $ABA^{-1}B^{-1}$ is a diagonal matrix. Then on 
introducing a transcendental element $y$ and setting $D=TAT^{-1}, C=TBT^{-1}$
for $T=\mbox{diag}(1,y)$, we have that $\langle A,B,C,D\rangle$ is
a faithful representation of the surface group $S_2$ in $SL(2,\F_p(x,y))$.
\end{co}
\begin{proof}
This is the case in the above proposition where $G_1=\langle a,b\rangle$ and 
$G_2=\langle d,c\rangle$ are both copies of the free group $F_2$, 
where $a=d$ and $b=c$. On setting $H$ to be the cyclic subgroup
generated by $aba^{-1}b^{-1}=dcd^{-1}c^{-1}$, we see that $G_1*_H G_2$ is 
$S_2$. We let $\rho_1$ send $a,b$ to $A,B$ and $\rho_2$ send $d,c$ to $A,B$,
thus (a) and (b) are satisfied. But (c) is satisfied too, because
an element $X$ of $SL(2,\F)$ with an off diagonal entry zero would have the
property that $\langle X,ABA^{-1}B^{-1}\rangle$ is a soluble group, which
which cannot happen in the free group on $A,B$ if
$X\notin \langle ABA^{-1}B^{-1}\rangle$.
\end{proof}

We now look for matrices of the required form.
\begin{thm} \label{matfrm}
Let $\F$ be any infinite field and let $c,h,d,\delta$ be non zero
elements of $\F$. Set $X=1-d\delta h+d^2 h^2$, $Y=\delta^2-d \delta h+h^2$
and suppose that $X$ and $Y$ are also non zero. Then on defining
\[
A=\sma{cc}\frac{dY}{X}&\frac{d\delta h(1-d^2)+d^2\delta^2-1}{cX}\\
c&d\fma
\mbox{ and }B=\sma{cc}\frac{\delta X}{Y}&
\frac{d\delta(1-\delta^2)+h(d^2\delta^2-1)}{cY}\\
ch&\delta\fma,\]
we have that $A,B\in SL(2,\F)$ with
$\mbox{tr}(A)=d(X+Y)/X$, $\mbox{tr}(B)=\delta (X+Y)/Y$
and $\mbox{tr}(AB)=(d\delta(1+h^2)-h)(X+Y)/(XY)$. We also have
\[AB=\sma{cc}\frac{d\delta(1+h^2)-h}{X}&
\frac{dh(\delta^2-1)+\delta(d^2-1)}{cX}\\
\frac{c(\delta+dh^3)}{Y}&\frac{d\delta(1+h^2)-h}{Y}
\fma,
BA=\sma{cc}\frac{d\delta(1+h^2)-h}{Y}&
\frac{dh(\delta^2-1)+\delta(d^2-1)}{cY}\\
\frac{c(\delta+dh^3)}{X}&\frac{d\delta(1+h^2)-h}{X}\fma\]
\[\mbox{ and }ABA^{-1}B^{-1}=\sma{ccc}\frac{Y}{X}&0\\0&\frac{X}{Y}\fma.\]
\end{thm}
\begin{proof}
Of course this can be established by direct calculation (preferably by
computational means). However we do indicate the derivation of our
expression for the above matrices but, as it will not be needed by those
willing to take the result on trust, we relegate it to the Appendix.
\end{proof}

It remains to be seen that we can find matrices $A,B$ of the above form
which generate a free group of rank 2 when $\F=\F_p(x)$.
To do this, we can take the 
discrete valuation 
$v$ on $\F_p(x)$ given by minus the degree, so 
$(a_m x^m+\ldots +a_0)/(b_n x^n+\ldots +b_0)$ has valuation $n-m$ if
$a_m,b_n\neq 0$. 
Recall that a discrete valuation $v:\F\rightarrow
\Z\cup\{\infty\}$ on a field $\F$ satisfies:\\
(1) $v(x)=\infty$ if and only if $x=0$\\
(2) $v(xy)=v(x)+v(y)$\\
(3) $v(x+y)\geq \mbox{min}(v(x),v(y))$. Moreover if $v(x)\neq v(y)$
then this is an equality.\\
The set of elements ${\cal O}_v=\{x\in \F:v(x)\geq 0\}$ forms a
subring of $\F$, called the valuation ring, which is a principal ideal
domain and an element $\pi$ with $v(\pi)=1$ is called a uniformiser.

We can then take the metric completion $k$ of
$\F_p(x)$ to obtain a local field, with the same valuation
and which also acts on its Bass - Serre tree
$T_{p+1}$. (The above is evaluation at zero: perhaps the more common
valuation used is that at infinity, giving $k=\F_p((x))$ but this is
obtained anyway by substituting $1/x$ for $x$ in everything below.)  
The results of \cite{con} then tell us when a pair of matrices
$A,B\in SL(2,k)$ generate a free and discrete group. A necessary condition
is that the valuation $v$ of the traces $A,B$ and $AB$ are all negative, as
otherwise these will act on $T_{p+1}$ as elliptic elements. Otherwise the
translation length of a matrix $M\in SL(2,k)$ is $-2v(\mbox{tr}\,M)$.
Now if we can find matrices $A,B$ in the above form where
the valuation of the three traces $\mbox{tr}(A),\mbox{tr}(B),\mbox{tr}(AB)$
are all equal and negative (say $-1$) then we are in Case 2(i) of
Proposition 3.5 in \cite{con}, with this satisfying the hypothesis in
Corollary 3.6 which shows that $\langle A,B\rangle$ is free of rank 2 and
discrete.
\begin{thm} \label{mat}
If $p$ is any odd prime then the following matrices 
$A,B,C,D$ in 
$SL(2,\F_p(x,y))$ generate a faithful representation of the genus 2 surface
group, where we have $ABA^{-1}B^{-1}=DCD^{-1}C^{-1}$.
\begin{eqnarray*}
A=\sma{cc} \frac{1-2x^2-2x^3}{x(x-1)}&\frac{-1+2x^2+x^3+x^4}{x^3(x-1)}\\
1&1/x^2\fma &,&
B=\sma{cc} \frac{x^2-1}{x-2x^3-2x^4}&
\frac{1+2x-x^2-3x^3-2x^4}{x^2(2x^3+2x^2-1)}\\x&1+x\fma,\\
D=\sma{cc} \frac{1-2x^2-2x^3}{x(x-1)}&\frac{-1+2x^2+x^3+x^4}{yx^3(x-1)}\\
y&1/x^2\fma &,&
C=\sma{cc} \frac{x^2-1}{x-2x^3-2x^4}&
\frac{1+2x-x^2-3x^3-2x^4}{yx^2(2x^3+2x^2-1)}\\yx&1+x\fma.
\end{eqnarray*}
\end{thm}
\begin{proof}
The form of $A$ and $B$ come from the previous
result with ($c=1$ without loss of generality and)
$d=1/x^2$, $\delta=x+1$ and $h=x$, whereupon
we do find (for $p\neq 2$) that the traces of $A,B,AB$ all have valuation
$-1$. Hence $\langle A,B\rangle$ is a rank 2 free group in the required
form for the application of Corollary \ref{shas}.

Although not needed for this proof, we briefly
indicate how $d,\delta,h$ were chosen. Looking at the matrices in Theorem
\ref{mat},
a necessary condition for $\langle A,B\rangle$ to be discrete and free is that
$ABA^{-1}B^{-1}$ is hyperbolic, so we require $v(Y)\neq v(X)$. On picking 
$v(X)=1$ and $v(Y)=-2$ (a somewhat ad hoc choice, obtained by examining a
specific case that was found to work by computation), if we want the valuation
of $\mbox{tr}(A)=d(X+Y)/X$ to be $-1$ under our choices for $v(X)$ and $v(Y)$
then this happens if and only if $v(d)=2$. Similarly the
valuation of $\mbox{tr}(B)=\delta(X+Y)/Y$ being $-1$ under the same condition
is equivalent to $\delta=-1$.
If we now try to satisfy $v(X)=v(1-d\delta h+d^2 h^2)=1$ with $v(d)=2$ and
$v(\delta)=-1$ then we cannot take $v(h)\geq 0$ because this implies that
$v(X)=0$. However $v(h)=-1$ would work
if there is some cancellation in forming $1-d\delta h$, each term of which has
valuation 0 but whose difference we will now assume has valuation $1$. 
These values
for $v(d),v(\delta),v(h)$ also give $v(Y)=-2$ as required if there is no  
cancellation when adding $\delta^2$ and $h^2$, both of which have valuation
$-2$.

As for obtaining $v(\mbox{tr}(AB))=-1$, this now happens if 
$v(d\delta+h(d\delta h-1))$ is 0. As we have assumed above one case of
cancellation in forming $d\delta h-1$, this will be true.

To ensure these conditions hold for specific choices of $d,\delta,h$, we aim
for the simplest expressions we can find. If $d\delta h=(x+1)/x$ then we
will have $v(d\delta h)=0$ but $v(d\delta h-1)=v(1/x)=1$ which gives us
the required cancellation. Thus on choosing
$\delta=x+1, h=x$ and therefore $d=1/x^2$, we see that $\delta^2+h^2=
2x^2+2x+1$, so if $p\neq 2$ then we do not have cancellation in forming this
sum and so all the required conditions above are satisfied.

Unfortunately this argument cannot work in $\F_2(x)$ because there will
always be cancellation when adding elements with the same valuation.
To deal with this case, we tried further possibilities for the valuations,
this time looking for the traces of $A,B,AB$ each to have valuation $-2$. 
On trying $v(d)=4,v(\delta)=-2,v(h)=-2$, we see that this would imply
$v(X)>0$ (because of the cancellation between 1 and $d\delta h$),
whereas $v(Y)>-4$ (because of cancellation between $\delta^2$ and $h^2$).
If we could arrange to have lots of cancellation between $d\delta h$ and 1
so that $v(d\delta h+1)=5$, giving
$v(X)=4$, and some cancellation between $\delta^2$ and $h^2$ giving
$v(\delta^2+h^2)=-2$ and thus $v(Y)=-2$,
then we obtain suitable traces for $A,B$ and $v(\mbox{tr}(AB))$ will also
be $-2$ if $v(Z)=1$, which will hold from imposing $v(d\delta h+1)=5$ above.

Thus we trying setting $\delta=x^2$ and $h=x^2+x+1$ to give us the
correct amount of cancellation in $\delta^2+h^2$. Then $d\delta h$ has to be 
something like $x^5/(x^5+1)$ to provide enough cancellation when adding it
to 1. Thus we also set $d=x^3/((x^2+x+1)(x^5+1))$ and these values satisfy
all of the above conditions, hence we have shown:
\end{proof}
\begin{co} \label{mat2}
The following matrices 
$A,B,C,D$ in 
$SL(2,\F_2(x,y))$ generate a faithful representation of the genus 2 surface
group, where we have $ABA^{-1}B^{-1}=DCD^{-1}C^{-1}$.
\begin{eqnarray*}
A&=&\sma{cc} 
\frac{x^8 + x^7 + x^5 + x^4 + x^3}{x^6 + x^5 + 1}&   
\frac{x^{13} + x^{11} + x^2 + x + 
    1}{(x^6 + x^5 + 1) (x^5 +1) (x^2 + x + 1)}\\
1&\frac{x^3}{(x^5 +1) (x^2 + x + 1)}\fma,\\
B&=&\sma{cc}
\frac{x^8 + x^7 + x^2}{(x^7+x^2+1) (x^5 + 1)}&\frac{x^{12} + x^{10} + x^9 + 
x^5 + x^4 + x^2+ 1}{(x^7 + x^2 + 1)(x^5 +1) (x^2 + x + 1)}\\
x^2 + x + 1&x^2\fma,
\end{eqnarray*}
and $D,C$ are the conjugates of $A,B$ respectively by $\mbox{diag}(1,y)$.
\end{co}

We now show that the necessary condition for a proper action on a finite 
product of locally finite trees, given in the previous section, is
satisfied by the surface group $S_g$.

\begin{co} \label{moreacc}
If $S_g$ is the closed orientable hyperbolic surface group of genus $g\geq 2$ 
then for any non identity element $\gamma\in S_g$ there is an action of $S_g$
on some locally finite tree which is minimal, faithful and such that $\gamma$
acts hyperbolically.
\end{co}
\begin{proof}
We show this for the group $S_2$ as all other $S_g$ are finite index
subgroups of $S_2$, so the action will still be minimal.

We regard the global field $\F_p(x)$ as sitting inside its metric completion,
the local field $k$, whereupon $SL(2,k)$ also acts on the $p+1$ regular
tree by automorphisms, and does so faithfully apart from $-I$.
Given the above matrices $A,B,C,D\in SL(\F_p(x,y))$, we can regard
them as lying in $SL(2,k)$ on taking $y$ to be any element in $k$ which is
transcendental over $\F_p(x)$ and such elements exist by countability
considerations. We thus obtain a faithful action of $S_2$ on the tree
$T_{p+1}$. Now the free subgroups $\langle A,B\rangle$ and $\langle C,D\rangle$
both act purely hyperbolically by construction, so our element $\gamma$
will automatically be a hyperbolic element if it is conjugate into either of
these subgroups. Otherwise, as mentioned at the end of the proof of
Proposition \ref{shl}, the trace of $\gamma$
will be a Laurent polynomial in the variable $y$ of the form 
$\alpha y^l$ plus lower order terms in $y$,
where $l>0$ and
$\alpha$ is a non zero element of $\F_p(x)$. Now this expression does not 
change if we change the element $y$ in $k$, as long as it is still 
transcendental over $\F_p(x)$.
To ensure that $\gamma$ is hyperbolic here, we require that its trace
has negative valuation. But if $y$ is a transcendental element of $k$ then
for $n\in\Z$ we have that $z=yx^n\in k$ is still transcendental over
$\F_p(x)$ and with $v(z)=v(y)-n$. Thus regardless of $v(\alpha)$ or the
other coefficients in the Laurent polynomial for $\mbox{tr}(\gamma)$, if
we take $n$ large enough and replace $y$ with $z$ then this trace will have
negative valuation. (In fact this argument allows us to ensure any finite
collection of non identity elements can all be made to act hyperbolically
in a single action.)

Finally we are not claiming that this action will definitely be minimal. But
if we restrict to its core, we have a minimal action in
which hyperbolic elements remain hyperbolic. Thus we can only lose
faithfulness of the action if there were some non identity element in $S_2$
which is acting elliptically on $T_{p+1}$ in the original action but
which acts as the identity when restricted to the core. However this action
of $SL(2,k)$ extends to an action on the boundary of the tree $T_{p+1}$, 
which is projective space $\mathbb P^1(k)$, where elements act as
M\"obius transformations. Such a transformation fixes
at most two ends, whereas if this element were acting as the identity on
the core, it would fix all ends of the core. But $A$ and $B$ provide two
independent hyperbolic elements in this action, thus the core has more than
two ends and hence this restriction is still faithful.
\end{proof}

We end by pointing out that, although we might regard the above as positive
evidence that a hyperbolic surface group acts properly on a finite product
of locally finite trees, we have not demonstrated the existence of a single
word hyperbolic group which has such such an action but which is not virtually
free. However, let us recall the infamous question credited to Gromov asking 
whether such a group always contains a surface group. If we had demonstrated
such an example then either it contains a surface subgroup $S_g$, thus $S_g$
and $S_2$ will also act properly on a finite product of locally finite
trees, or we would have a counterexample to this question. Looking at it
the other way, if it could be shown that $S_g$ does not have such an action
then nor can any word hyperbolic group (other than those which are virtually
free), at least if Gromov's question has a positive answer.

\section{Appendix}

As mentioned in Theorem \ref{matfrm}, we show:
\begin{thm}
Suppose $\F$ is any infinite field (in any characteristic) and
$A,B$ are matrices in $SL(2,\F)$ with $ABA^{-1}B^{-1}$ a diagonal matrix.
Then, apart from some exceptional cases where $\langle A,B\rangle$ cannot
be free of rank 2 and possibly one further family of representations,
there exist non zero elements $c,d,\delta,h\in \F$ such
that if we set $X=1-d\delta h+d^2 h^2$ and $Y=\delta^2-d\delta h+h^2$ then
$X$ and $Y$ are also non zero and
\[
A=\sma{cc}\frac{dY}{X}&\frac{d\delta h(1-d^2)+d^2\delta^2-1}{cX}\\
c&d\fma\quad,\quad B=\sma{cc}\frac{\delta X}{Y}&
\frac{d\delta(1-\delta^2)+h(d^2\delta^2-1)}{cY}\\
ch&\delta\fma,\]
with $ABA^{-1}B^{-1}=\mbox{diag}(Y/X,X/Y)$. Conversely if $A,B$ are
in the above form with $c,d,\delta,h,X,Y$ all non zero then $ABA^{-1}B^{-1}$
is equal to the diagonal matrix $\mbox{diag}(Y/X,X/Y)$.
\end{thm}
\begin{proof}
We first note that if a nondiagonal entry of $M=A,B,AB$ or $BA$ is zero then
the subgroup $\langle M,ABA^{-1}B^{-1}\rangle$ is soluble and so 
$\langle A,B\rangle$ cannot be $F_2$. Also if a diagonal entry of $M$ as above
is zero then the same holds for 
$\langle MABA^{-1}B^{-1}M^{-1},ABA^{-1}B^{-1}\rangle$, so again $A$ and $B$ do not
generate $F_2$.

Consider the action of $SL(2,\F)$ on the projective space $\mathbb{P}^1(\F)$
via M\"obius transformations. Any non identity element has at most 2 fixed
points and for $ABA^{-1}B^{-1}$ these are $\infty$ and $0$. Consequently if
$A^{-1}B^{-1}$ sends $\infty,0$ to $x,y\in\mathbb{P}^1(\F)$ respectively
then both $AB$ and $BA$ send $x,y$ to $\infty,0$. Now we have $x\neq y$
and by the point above about zeros in the matrices, we also see that
neither $x$ nor $y$ is $\infty$ or $0$, so we can regard $x,y$ as distinct
elements of $\F^*$. This means that there are other elements $s,t,u,v\in\F^*$
such that 
\[AB=\left( \begin{array}{rr}
u & -uy \\ s & -sx \end{array} \right) \quad\mbox{and}\quad
BA=\left( \begin{array}{rr}
v & -vy \\ t & -tx \end{array} \right). \]

But $AB$ and $BA$ are conjugate matrices, so we can equate their trace
and determinant. The latter gives us $us=vt$ (as $x\neq y$) and then the
former implies that $v=-sx$ (as $t=s$ gives us $A=B$) and $u=-tx$.
Moreover as both of these determinants are 1, we now have
\[
AB=\sma{rc} -tx&tx^2-1/s\\s&-sx\fma,
BA=\sma{rc} -sx&sx^2-1/t\\t&-tx\fma\]
with $ABA^{-1}B^{-1}=\mbox{diag}(t/s,s/t)$.

We are now ready to set 
\[A=\left( \begin{array}{rr}
a & b \\ c & d
\end{array} \right) \quad\mbox{and}\quad
B=\left( \begin{array}{rr}
\alpha& \beta \\ \gamma & \delta
\end{array} \right)\]
with $ad-bc=1=\alpha\delta-\beta\gamma$ and we equate
some coefficients in $AB$ and $BA$. The $-tx$ coefficient being the same in
both $AB$ and $BA$ is equivalent to $a\alpha=d\delta$, which is also the
same condition for the $-sx$ coefficients to match. As $s,t\neq 0$ we will
set $r=s/t$, whereupon taking the ratio of the bottom left hand entries of 
$AB$ and $BA$ gives us that
\[r=\frac{c\alpha+d\gamma}{c\delta+d\delta\gamma/\alpha}=\alpha/\delta.\]
Thus we now replace
$a$ and $\alpha$ with $d/r$ and $r\delta$ respectively. At this point,
we also take the opportunity to replace $b$ by $(d^2-r)/(rc)$ and
$\beta$ by $(r\delta^2-1)/\gamma$. Moreover as we can conjugate $A$ and $B$
by diagonal matrices without altering $ABA^{-1}B^{-1}$, we can multiply $c$
and $\gamma$  by the same arbitrary non zero constant. We could impose $c=1$
but we will set $h=\gamma/c$. Substituting all of the above now gives us
\[A=\left( \begin{array}{cc}
\frac{d}{r}& \frac{d^2-r}{rc}\\ c & d
\end{array} \right) \quad,\quad
B=\left( \begin{array}{cc}
r\delta & \frac{r\delta^2-1}{ch} \\ ch & \delta
\end{array} \right). \;\;\,\quad\]
From the argument above, we have that the respective diagonal elements of 
$AB$ and $BA$ match and $r$ is the ratio of the lower left hand entries.
As the upper right hand entries are forced by the determinants being 1,
we now have that $AB$ and $BA$ are in the correct form (for some
$s,t,x\in\F^*$) if and only if the ratio of the bottom right hand entries
of $AB$ and $BA$ is $s/t=r$, so reading this off from $A$ and $B$ gives us
\[r=\frac{\frac{r\delta^2-1}{h}+d\delta}{\frac{(d^2-r)h}{r}+d\delta}.\]
This is actually linear in $r$ and results in the equation $Yr=X$ for
$X=1-d\delta h+d^2h^2$ and $Y=\delta^2-d\delta h+h^2$.

If $X$ and hence $Y$ are non zero then $A$ and $B$ have the form as
given in the statement of the theorem and we know $ABA^{-1}B^{-1}$
is $\mbox{diag}(1/r,r)$, so this is equal to $\mbox{diag}(Y/X,X/Y)$.
(If $X$ is zero, hence so is $Y$ as $r\neq 0$, then there can be one
further family of solutions as mentioned in the statement of the theorem.
Namely it can be checked by direct calculation that this situation
where $X=Y=0$ is equivalent to ($c=1$ without loss of generality and)
$d^2=-1/(h^2(h^2+1)),\delta=-dh^3,r=h(h^2+1)$ but over $\R$ for instance
this would not give any further solutions.)

Conversely suppose that $A$ and $B$ have the form above for some elements
$c,d,\delta,h$ and $X=1-d\delta h+d^2 h^2$, $Y=\delta^2-d\delta h+h^2$
of $\F$, all of which are non zero. Then reversing the argument above
gives us that $AB$ and $BA$ are of the required form to ensure that
$AB(BA)^{-1}$ is a diagonal matrix, with the diagonal entries given by
the ratio of the bottom right hand entries of $AB$ and $BA$, thus
$r=X/Y$ and $1/r=Y/X$.
\end{proof}

\end{document}